\documentclass[10pt,twoside]{article}

\usepackage[english]{babel}
\usepackage[utf8]{inputenc}
\usepackage{amsmath}
\usepackage{amsfonts}
\usepackage{mathtools}
\usepackage{amssymb}
\usepackage{amsthm}
\usepackage{hyperref}
\usepackage{graphicx}

\usepackage{xcolor}

\theoremstyle{definition} 
\newtheorem{defi}{Definition}[section]
\newtheorem{remark}[defi]{Remark}

\theoremstyle{plain}
\newtheorem{lemma}[defi]{Lemma}
\newtheorem{theorem}[defi]{Theorem}
\newtheorem{corollary}[defi]{Corollary}

\newcommand{\myproduo}[2]{\prod\limits_{#1}^{#2}}
\newcommand{\myspan}{\text{span}\;}

\newcommand{\myotimes}[2]{\bigotimes\limits_{#1}^{#2}}

\numberwithin{equation}{section} 

\begin{document}


\pagestyle{myheadings}


\title{Product kernels are efficient and flexible tools for high-dimensional scattered data interpolation}

\date{\today}

\author{\large
Kristof Albrecht\footnote{Universit\"at Hamburg, Dept of Mathematics, {\tt kristof.albrecht@studium.uni-hamburg.de}} \qquad
Juliane Entzian\footnote{Universit\"at Hamburg, Department of Mathematics, {\tt juliane.entzian@uni-hamburg.de}} \qquad
Armin Iske\footnote{Universit\"at Hamburg, Department of Mathematics, {\tt armin.iske@uni-hamburg.de}}
}

\markboth{\footnotesize \rm \hfill K.~ALBRECHT, J.~ENTZIAN, AND A.~ISKE \hfill}
{\footnotesize \rm \hfill PRODUCT KERNELS ARE EFFICIENT AND FLEXIBLE TOOLS\hfill}

\maketitle
\thispagestyle{plain}


\begin{abstract}
	This work concerns the construction and characterization of product kernels for multivariate approximation from a finite set of discrete samples.
	To this end, we consider composing different component kernels, each ac\-ting on a low-dimensional Euclidean space.
	Due to Aronszajn (1950), the product of positive {\em semi-}definite kernel functions is again positive {\em semi-}definite, 
	where, moreover, the corresponding native space is a particular instance of a tensor product, referred to as Hilbert tensor product.

	We first analyze the general problem of multivariate interpolation by product kernels. 
	Then, we further investigate the tensor product structure, in particular for {\em grid-like} samples.
         We use this case to show that the product of positive definite kernel functions is again positive definite. 
	Moreover, we develop an efficient computation scheme for the well-known Newton basis. 
	Supporting numerical examples show the good performance of product kernels, especially for their flexibility.
\end{abstract}

\section{Introduction}\label{intro}
Positive definite functions are powerful tools for multivariate interpolation. Commonly used kernels, such as Gaussians, (inverse) multiquadrics and polyharmonic splines, are well-known to provide high-performance reconstructions in scattered data approximation~\cite{wendland}. Due to Schaback's uncertainty relation~\cite{schaback1995error}, however, their resulting reconstruction methods are often critical when it comes to combine high order approximation with numerical stability. In many cases, this leads to severe limitations, e.g.~in particle-based fluid flow simulations, where {\em anisotropic} particle samples are very critical from a numerical viewpoint.
In fact, applications in computational fluid dynamics require more sophisticated constructions of more efficient and highly flexible kernel methods. 

In our previous work~\cite{anisoPart_AlbEntIsk}, we considered working with kernels on anisotro\-pic transformations to better adapt the kernel reconstruction scheme to the spatial distribution of sample points. In that case, anisotropic kernels are constructed by replacing the Euclidean norm in the argument of a fixed radial kernel 
by an {\em anisotropic} norm $\|\cdot\|_A$, $A$-norm, defined as
\begin{equation*}
\Vert x \Vert_A^2 := x^T A x
\qquad \mbox{ for }  x \in \mathbb{R}^d,
\end{equation*}
where $A\in\mathbb{R}^{d \times d}$ is a suitably chosen symmetric positive definite matrix. 
To make one example, for the standard Gaussian kernel and a diagonal matrix $A = \text{diag} (\varepsilon_1, ... , \varepsilon_d)$ with positive diagonal entries 
$\varepsilon_i>0$, this yields the kernel
\begin{equation*}
e^{-\|x\|^2_A} = e^{-x^T A x} = \prod_{i=1}^{d}e^{- \varepsilon_i x_i^2} \qquad \text{for } x = \left( x_1, ... , x_d \right)^T \in \mathbb{R}^d,
\end{equation*}
see \cite{Fasshauer2011}. Hence, this anisotropic version of the standard Gaussian kernel is given by a product of $d$ univariate kernels, each of them equipped with an own shape parameter $\varepsilon_i > 0$, $i=1,\dots,d$. 

This gives rise to study kernels that are pro\-ducts of positive definite functions defined on different lower-dimensional spaces, 
which are not necessarily translational-invariant.
The main intention of this approach is to further improve the flexibility of kernel-based reconstruction methods, as the initial domain can be split into several axes and each axis can be equipped with an individual kernel function depending on the application. 
In this paper, we will further elaborate on the concept of product kernels in order to construct more flexible kernel reconstruction schemes,
whose performance we then compare with reconstructions by standard kernels. 

The outline of this paper is as follows.
In Section \ref{prodKernel}, we introduce pro\-duct kernels as special types of positive semi-definite functions.
We discuss some of their properties, where we mainly rely on the work of Aronszajn~\cite{aronszajn1950theory}, 
who showed that native spaces of product kernels are tensor products of the corresponding components' native spaces, 
with a well-behaved inner product.
This type of tensor product is referred to as the \textit{Hilbert tensor product}~\cite{KadisonRingrose1983,Neveu1968}. 
In contrast to the standard tensor product, the Hilbert tensor product of Hilbert spaces yields a Hilbert space. Furthermore, the inner product of tensors is given by the multiplication of the inner products in the component spaces. 
We explain the basic theory on product kernels and their native spaces in Section~\ref{sec:Tensor_Hilbert}.

In Section \ref{interpolation}, we discuss interpolation by product kernels. 
To this end, the 
positive definiteness of a kernel is desired for the well-posedness of the kernel's interpolation scheme.
We remark that product kernels from finitely many 
positive definite {\em translation-invariant} kernels are 
positive definite (cf.~\cite[Proposition 6.25]{wendland}). To the best of our knowledge, however, the more general case of product kernels from {\em arbitrary} 
positive definite kernels has not been covered, yet.
To tackle this open problem, we consider \textit{grid-like} sets of samples (i.e., grid-like sets of interpolation points).
We prove that the product interpolation matrix on grid-like point sets is the Kronecker pro\-duct of its component interpolation matrices. 
This important observation enables us to show that product kernels maintain the 
positive definiteness of their component kernels, 
see Theorem~\ref{thm:posdef}.

Due to the structure of the inner product in Hilbert tensor product spaces, it is possible to efficiently compute orthonormal bases in the case of grid-like point sets. We construct the \textit{Newton basis} (cf.~\cite{Muller2009article,Pazouki2011}) by combining the pre-computed Newton bases of the component spaces via the Kronecker product, which reduces the computational costs significantly, as we show in Section~\ref{Newton_basis}.

In Section~\ref{sec:greedy}, we employ that Newton basis to construct an efficient update formula for the proposed interpolation scheme.
This is done by point selection strategies that essentially maintain the grid-like structure of the sample points.
Our discussion in Section~\ref{sec:greedy} results in a component-wise version of the \textit{P-greedy algorithm}~\cite{DeMarchi2005}, 
which converges for continuous kernels on compact domains.

In Section \ref{numeric}, we finally provide supporting numerical examples to show the flexibility of product kernels.
In particular, the good performance of product kernels on grid-like data demonstrates their utility.

\section{Product kernels} \label{prodKernel}
Kernel-based approximation aims at finding an interpolant $s_f$ to a target function $f$ satisfying \textit{interpolation conditions}
\begin{equation}\label{eq:interpocond}
s_f(x_i) = f(x_i)
\qquad \mbox{ for } i=1,\ldots,n,
\end{equation}
on given pairwise distinct data points $X= \{ x_1,\ldots, x_n \} \subset \mathbb{R}^d$
and corres\-ponding function values 
$f_X = \left( f(x_1), ... , f(x_n) \right)^T \in \mathbb{R}^n$.
For a fixed kernel function $K: \mathbb{R}^d \times \mathbb{R}^d \longrightarrow \mathbb{R}$, we restrict the interpolant to the form 
\begin{equation} \label{eq:interpolant_form}
s_f = \sum_{i=1}^{n} c_i\; K(\cdot, x_i),
\end{equation}
containing the coefficients $c = (c_1,\ldots,c_n)^T \in \mathbb{R}^n$. Inserting (\ref{eq:interpolant_form}) into (\ref{eq:interpocond}) leads to the linear system
\begin{align} \label{eq:intpol_system}
	A_{K,X} \cdot c = f_X, 
\end{align}
with the \textit{interpolation matrix}
\begin{align*}
 A_{K , X} := \left( K(x_i , x_j) \right)_{1\leq i,j \leq n}\in\mathbb{R}^{n\times n}.
\end{align*}
The kernel $K$ is called \textit{positive \mbox{(semi-)definite}}, iff for any finite set of pairwise distinct data points $X= \{ x_1,\ldots, x_n \} \subset \mathbb{R}^d$, $n \in \mathbb{N}$,
the interpolation matrix $A_{K,X}$ is symmetric and positive (semi-)definite. For the well-posedness of the interpolation scheme, $K$ is required to be positive defi\-nite, so that the system \eqref{eq:intpol_system} has a unique solution. For more details on positive (semi)-definite functions and their properties concerning scattered data interpolation we refer to~\cite{Iske_Approx,wendland}.

In the following of this section, we introduce product kernels and summarize a few relevant results. To this end, we consider the (finite) Cartesian product
$$
     \mathbb{R}^d \simeq \mathbb{R}^{d_1} \times \dots \times \mathbb{R}^{d_M}
     \qquad \mbox{ for } d_i \in \mathbb{N},  \; i=1,\dots,M,
$$ 
between lower dimensional Euclidean spaces $\mathbb{R}^{d_i}$, so that $d = \sum_{i=1}^{M} d_i$. 
For ${x\in\mathbb{R}^d}$, its projection on the $i$-th product space $\mathbb{R}^{d_i}$ is denoted as $x^i$. 
Likewise, for the projection of sets $X = \{ x_1,\ldots, x_n \} \subset \mathbb{R}^d$ onto 
$\mathbb{R}^{d_i}$, we use the notation 
$$
    X^i := \left\{ x_1^i , \ldots , x_n^i \right\} \subset \mathbb{R}^{d_i}
    \qquad \mbox{ for } i = 1,\ldots, M.
$$
It is important to remark at this point, that the entries of $X^i$ may contain duplicate elements.
We will further discuss this possibility later in Section~\ref{interpolation}.

\begin{defi} \label{def:prodKernel}
	Let $K_i:\mathbb{R}^{d_i} \times \mathbb{R}^{d_i} \longrightarrow \mathbb{R}$ be positive (semi-)definite kernels on $\mathbb{R}^{d_i}$, ${i=1,\dots,M}$ and $d := \sum_{i=1}^{M} d_i$. Then the {\tt product kernel}
	\begin{align*}
		K : \mathbb{R}^{d} \times \mathbb{R}^{d} \longrightarrow \mathbb{R}
	\end{align*}
	is defined as
	\begin{align*}
	K(x,y):= \myproduo{i=1}{M} K_i(x^i,y^i) \qquad \text{for } x,y \in \mathbb{R}^d.
\end{align*}
The functions $K_1, \ldots, K_M$ are called the {\tt component kernels} of $K$, and we use the notation $${K = \prod_{i=1}^{M} K_i}.$$
\qed
\end{defi}

\begin{figure}[h!]
	\begin{center}
		\fbox{
			\includegraphics[width=11cm]{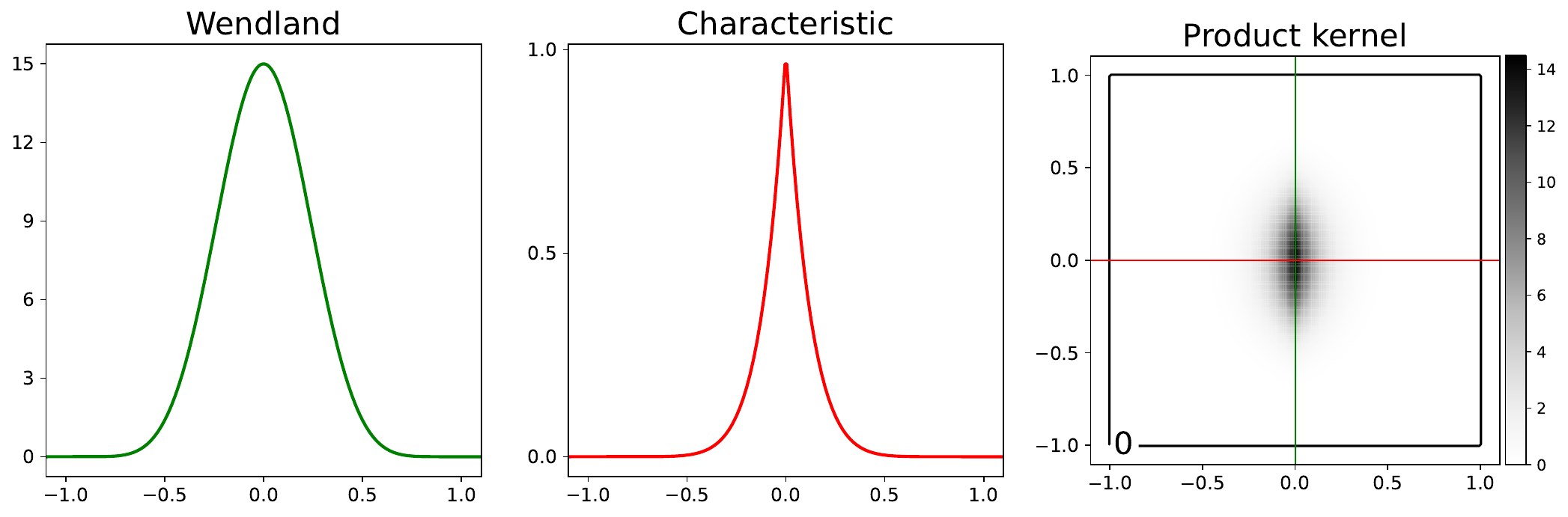}
		}
	\end{center}
	\caption{Bivariate product kernel: Visualization of the univariate component kernels (left, middle) and top view on the product kernel (right).}	
	\label{fig:productKernel}
\end{figure}

Figure~\ref{fig:productKernel} visualizes the construction of a two-dimensional product kernel. 
On given data points $X = \{ x_1,\ldots, x_n \} \subset \mathbb{R}^d$
and a product kernel $K$ with components $K_i$, 
the entries of the corresponding interpolation matrix can be written as
\begin{align*}
	\left(A_{K,X}\right)_{j,k} = K(x_j,x_k)
	= \myproduo{i=1}{M}\; K_i(x_j^i,x_k^i)
	= \myproduo{i=1}{M}\; \left(A_{K_i,X^i}\right)_{j,k},
\end{align*}
so that $A_{K,X}$ is the \textit{Hadamard product}~\cite[Chapter 5]{Horn_MatrixAna} of the component interpolation matrices $A_{K_i,X^i}$. 
The next result from~\cite{aronszajn1950theory} relies on the \textit{Schur product theorem}~\cite[Theorem 5.2.1]{Horn_MatrixAna}.

\begin{theorem} \label{thm:positive_semidefinite}
	Let $K_i$ be positive semi-definite kernels on $\mathbb{R}^{d_i}$, $i=1,\dots,M$. 
	Then, the product kernel $K = \prod_{i=1}^{M} K_i$ is positive semi-definite on $\mathbb{R}^d$, where 
	$d = \sum_{i=1}^{M} d_i$.
	\qed
\end{theorem}

We remark that for a given finite set $X \subset \mathbb{R}^d$ of {\em pairwise distinct} points, it cannot be guaranteed that its projections 
$X^i \subset \mathbb{R}^{d_i}$ onto $\mathbb{R}^{d_i}$ contain {\em pairwise distinct} points, for $i=1,\dots,M$.
A simple counter\-example is given by a grid in $\mathbb{R}^d$. Hence, in this case, the component interpolation matrices are not positive definite. 
This is critical insofar as positive definiteness is desired for unique interpolation. We come back to this important point in Section~\ref{interpolation}.

From now, we assume that the components $K_i$, for $i=1,\ldots,M$, are positive semi-definite kernels, so that their product kernel is po\-si\-tive semi-definite according to Theorem \ref{thm:positive_semidefinite}. It is well-known that any positive semi-definite kernel $K$ generates a reproducing kernel Hilbert space (RKHS) $\mathcal{H}_{K,\Omega}$ of functions
on a fixed domain $\Omega \subset \mathbb{R}^d$ (cf.~\cite[Theorem 3]{berlinet2011reproducing}), called \textit{native space}, 
whose inner product we denote as $\langle \cdot, \cdot \rangle_{K}$, whereby the native space norm is defined as $\|\cdot \|_K = \langle \cdot, \cdot \rangle^{1/2}_{K}$. 
As the product kernel is composed of the different component kernels, we aim to derive a similar relation between the associated native spaces. 
To this end, we can rely on a result by Aronszajn~\cite{aronszajn1950theory} for two component kernels, which can directly be generalized to an arbitrary number of finitely many components,
 cf.~Theorem~\ref{thm:products}. This key result allows us to identify the native space of a product kernel as a special type of tensor product.
 Details on this are discussed in Section~\ref{sec:Tensor_Hilbert}.

In the following of this paper, we use the notation
\begin{align*}
 \mathcal{S}_{K,\Omega} := {\rm span} \Big\lbrace K(\cdot , x) \; \Big\vert \; x \in \Omega \Big\rbrace
 \qquad \mbox{ for } \Omega \subset \mathbb{R}^d.
\end{align*}
Recall that $\mathcal{S}_{K,\Omega}$ lies dense in the native space $\mathcal{H}_{K,\Omega}$ of $K$ on $\Omega$.
We can formulate this important result as follows (for further details on the general construction of the native space we refer to~\cite[Chapter 10.2]{wendland}).

\begin{theorem} \label{thm:products}
	Let $\Omega_i \subset \mathbb{R}^{d_i}$ and $K_i$ be a positive semi-definite kernel on $\mathbb{R}^{d_i}$ for $i=1,...,M$. Let $\Omega = \bigtimes_{i=1}^{M} \Omega_i \subset \mathbb{R}^d$ and $K = \prod_{i=1}^{M} K_i$. Then, the mapping
	\begin{align} \label{eq:multilinear_map}
		\varphi : \bigtimes\limits_{i=1}^{M} \mathcal{H}_{K_i,\Omega_i} \longrightarrow \mathcal{H}_{K,\Omega} 
		\qquad \left( f_1, \ldots , f_M \right) \longmapsto \prod\limits_{i=1}^M f_i
	\end{align}
	is a well-defined multilinear map that satisfies the identity
	\begin{align} \label{eq:innerprod_tensor}
		\langle \varphi(f_1, \ldots , f_M) , \varphi(g_1, ... , g_M) \rangle_K = \prod\limits_{i=1}^{M} \langle f_i, g_i \rangle_{K_i}
	\end{align}
	for all $(f_1 , \ldots , f_M),(g_1, \ldots , g_M) \in \bigtimes_{i=1}^{M} \mathcal{H}_{K_i,\Omega_i}$. Moreover, we have
	\begin{align} \label{eq:density}
		\mathcal{S}_{K,\Omega} \subset {\rm span} \lbrace \varphi(f_1, \ldots , f_M) \mid f_i \in \mathcal{H}_{K_i,\Omega_i} \mbox{ for } i=1, \ldots ,M \rbrace.
	\end{align}
	\qed
\end{theorem}
 
\section{Native spaces of product kernels} \label{sec:Tensor_Hilbert}
In this section, we introduce the \textit{Hilbert tensor product}, where relevant of its properties are first provided for preparation regarding our further analysis. 
Next we fix definitions and notations (cf.~\cite{KadisonRingrose1983}).

\begin{defi} \label{def:HS_mapping}
Let $\mathcal{H}_1, \ldots , \mathcal{H}_M,Z$ be real Hilbert spaces and
\begin{align*}
  \varphi: \bigtimes_{i=1}^M \mathcal{H}_i \longrightarrow Z
\end{align*}
be a multilinear function on the standard Cartesian product of $\mathcal{H}_1, \ldots , \mathcal{H}_M$.
\begin{itemize}
  \item The multilinear function $\varphi$ is {\tt bounded}, iff there is a $c \in \mathbb{R}_+$ satisfying
  \begin{align*}
    \Vert \varphi (x_1, \ldots ,x_M) \Vert_Z \leq c \cdot \prod\limits_{i=1}^M \Vert x_i \Vert_{\mathcal{H}_i}
    \qquad \mbox{ for all }  (x_1, \ldots , x_M) \in \bigtimes_{i=1}^M \mathcal{H}_i.
  \end{align*}
  \item 
  A bounded $\varphi$ is called {\tt weak Hilbert-Schmidt mapping}, iff there is a constant $d \in \mathbb{R}_+$, 
  such that for any sequence of orthonormal bases 
  $B_1 \subset \mathcal{H}_1,\ldots,B_M \subset \mathcal{H}_M$ and $z \in Z$ we have
  \begin{align*}
    \sum\limits_{b_1 \in B_1} \hdots \sum\limits_{b_M \in B_M} \vert \langle \varphi(b_1, \ldots ,b_{M}) , z \rangle_Z \vert^2 \leq d^2 \Vert z \Vert_Z^2.
  \end{align*}  
\end{itemize}
\qed
\end{defi}

In contrast to the standard tensor product, the Hilbert tensor product (cf.~Definition~\ref{def:Hilbert:tensor:product}) of Hilbert spaces is a Hilbert space. 
This is a crucial detail, since kernel-based approximation theory usually works with (reproducing kernel) Hilbert spaces. 
The following theorem is a slight modification of \cite[Theorem~2.6.4]{KadisonRingrose1983}, 
which details the main properties of Hilbert tensor products.


\begin{theorem} \label{thm:hilbert_tensor}
  For real Hilbert spaces $\mathcal{H}_1, \ldots , \mathcal{H}_M$, the following properties hold.
  \begin{enumerate}
    \item[(a)] 
    There is a multilinear map $\varphi: \bigtimes_{i=1}^M \mathcal{H}_i \longrightarrow \mathcal{H}$
    to a real Hilbert space $\mathcal{H}$,
    such that for any $(x_1, \ldots , x_M ),(y_1 , \ldots , y_M) \in \bigtimes_{i=1}^M \mathcal{H}_i$, we have
    \begin{equation}
    \label{inner_product}
      \langle \varphi (x_1, \ldots , x_M) , \varphi(y_1, \ldots ,y_M ) \rangle_\mathcal{H} = \prod\limits_{i=1}^M \langle x_i , y_i \rangle_{\mathcal{H}_i}
    \end{equation}
    and such that
    \begin{align*}
      \mathcal{H}_0 := \myspan \lbrace \varphi (x_1 , \ldots , x_M ) \mid ( x_1 , \ldots , x_M ) \in \mathcal{H}_1 \times \cdots \times \mathcal{H}_M \rbrace
    \end{align*}
    is a dense subset of $\mathcal{H}$.
    \item[(b)] 
    Let $B_1 \subset \mathcal{H}_1, \ldots , B_M \subset \mathcal{H}_M$ be orthonormal bases in $\mathcal{H}$. 
    Then
    \begin{align*}
      B = \lbrace \varphi(b_1, \ldots ,b_M) \mid (b_1, \ldots ,b_M) \in B_1 \times \cdots \times B_M \rbrace
    \end{align*}
    is an orthonormal basis of $\mathcal{H}$.
    \item[(c)] 
    The map $\varphi$ is a weak Hilbert-Schmidt mapping, satisfying the following universal property: 
    If $Z$ is a Hilbert space and $\psi: \bigtimes_{i=1}^M \mathcal{H}_i \longrightarrow Z$ is a weak Hilbert-Schmidt mapping, 
    then there is a unique bounded linear map $T: \mathcal{H} \longrightarrow Z$ such that $\psi = T \circ \varphi$.
    \item[(d)] 
    Let $\tilde{\mathcal{H}}$ be a Hilbert space and $\tilde{\varphi}:\bigtimes_{i=1}^M \mathcal{H}_i \longrightarrow \tilde{\mathcal{H}}$ be a weak Hilbert-Schmidt mapping satisfying the universal property in (c). Then, there exists an isometric isomorphism $U: \mathcal{H} \longrightarrow \tilde{\mathcal{H}}$ with $U \circ \varphi = \tilde{\varphi}$. Hence, the pair $(\tilde{\mathcal{H}},\tilde{\varphi})$ satisfies the properties in~(a).
    \item[(e)] 
    If the tuple $(\tilde{\mathcal{H}},\tilde{\varphi})$ satisfies the properties in~(a), 
    then there is an isometric isomophism $U: \mathcal{H} \longrightarrow \tilde{\mathcal{H}}$ with $U \circ \varphi = \tilde{\varphi}$.
  \end{enumerate}
\qed
\end{theorem}

We remark that the relation \eqref{inner_product} between the inner product $\langle \cdot, \cdot \rangle_{\mathcal{H}}$ in $\mathcal{H}$ and those in $\mathcal{H}_i$, $i=1,\ldots,M$,
implies that $\varphi$ is bounded with $c = 1$.
Moreover, property (b) in Theorem~\ref{thm:hilbert_tensor} implies that $\varphi$ is a weak Hilbert-Schmidt mapping with $d = 1$.

We continue to provide further definitions and notations.

\begin{defi}
\label{def:Hilbert:tensor:product}
  Let $\mathcal{H}_1, \ldots ,\mathcal{H}_M$ be Hilbert spaces and $(\mathcal{H},\varphi)$ satisfy the properties from Theorem \ref{thm:hilbert_tensor}~(a). 
  Then $(\mathcal{H},\varphi)$ is called the {\tt Hilbert tensor product} of $\mathcal{H}_1, \ldots,\mathcal{H}_M$, denoted by
  \begin{align*}
    (\mathcal{H},\varphi) = \bigotimes\limits_{i=1}^M \mathcal{H}_i, \qquad \text{in short:} \quad \mathcal{H} = \bigotimes\limits_{i=1}^M \mathcal{H}_i.
  \end{align*}
  Given $(x_1, \ldots ,x_M) \in \bigtimes_{i=1}^M \mathcal{H}_i$, the corresponding {\tt tensor} is written as
  \begin{align*}
    \underset{i=1}{\overset{M}{\otimes}} x_i = x_1 \otimes \cdots \otimes x_M := \varphi(x_1, \ldots,x_M).
  \end{align*}
  \qed
\end{defi}

Due to Theorem \ref{thm:hilbert_tensor}~(e), the Hilbert tensor product is, up to isometric isomorphy, 
uniquely characterized by the properties of Theorem \ref{thm:hilbert_tensor}~(a). 

Next we identify the native space of a product kernel $K = \prod_{i=1}^M K_i$ as the Hilbert tensor product of its component native spaces $\mathcal{H}_{K_i, \Omega_i}$.

\begin{theorem} \label{thm:native_tensor}
Let $\Omega_i \subset \mathbb{R}^{d_i}$ and $K_i$ be positive semi-definite kernels on $\mathbb{R}^{d_i}$ for $i=1,...,M$. 
Moreover, let $\Omega = \bigtimes_{i=1}^{M} \Omega_i \subset \mathbb{R}^d$ and $K = \prod_{i=1}^{M} K_i$. Then
\begin{align*}
  \mathcal{H}_{K, \Omega} = \bigotimes \mathcal{H}_{K_i, \Omega_i}.
\end{align*}
\end{theorem}

\begin{proof}
  Consider the multilinear mapping (\ref{eq:multilinear_map}) from Theorem \ref{thm:products}. According to (\ref{eq:innerprod_tensor}) and (\ref{eq:density}), 
  and since ${\mathcal S}_{K,\Omega}$ is dense in ${\mathcal H}_{K,\Omega}$,  
  the pair $(\mathcal{H}_{K,\Omega},\varphi)$ satisfies the properties in Theorem \ref{thm:hilbert_tensor}~(a). This completes the proof already by part~(e) of Theorem \ref{thm:hilbert_tensor}.
\end{proof}

\begin{remark}
We remark that the identification of the product kernel’s native space as a Hilbert tensor product can be found in~\cite[Chapter VI]{Neveu1968}, 
where the image of $\varphi$ from Theorem \ref{thm:products} is completed to an Hilbert space. 
But then, the (abstract) elements of this completion need to be converted into functions. 
Note that we can skip this conversion step in our analysis, since we already know that the product kernel generates a Hilbert space of functions and $\varphi$ maps directly into this function space, as explained in the preliminaries of Section~\ref{prodKernel}.
\end{remark}

\section{Interpolation with product kernels} \label{interpolation}
Recall from Section \ref{prodKernel} that the Hadamard product is not suitable to show that positive definite component kernels yield a positive definite product kernel.
We remark that the statement of Theorem~\ref{thm:positive_semidefinite} relies on 
the \textit{Schur product theorem}.

Another tool for the construction of positive semi-definite kernels is given by \textit{Bochner's theorem} (cf.~\cite[Theorem 6.6]{wendland}).
Recall that the characterization of Bochner's theorem makes a link between {\em translation-invariant} kernels $K$ and non-negative Fourier transforms,
by assuming the form
\begin{equation}
\label{bocher:assume:form}
    K(x, y) = \Phi(x - y) \qquad \text{ for } x, y \in \mathbb{R}^{d}.
\end{equation}

For the reader's convenience, we now formulate an extension of Bochner's theorem from kernels to product kernels (cf.~\cite[Proposition 6.25]{wendland}).

\begin{theorem} \label{thm:prodfunc_multi}
	Let $K_i$ be positive definite kernels on $\mathbb{R}^{d_i}$ of the form
	\begin{align*}
		K_i(x^i, y^i) = \Phi_i( x^i - y^i) \qquad \text{ for } x^i, y^i \in \mathbb{R}^{d_i},
	\end{align*}
	where $\Phi_i \in L_1(\mathbb{R}^{d_i}) \cap \mathcal{C}(\mathbb{R}^{d_i})$, for $i=1,\dots,M$.
	Then, the product kernel $$K = \prod_{i=1}^M K_i$$ is a positive definite kernel on $\mathbb{R}^{d}$, where $d=\sum_{i=1}^{M} d_i$.
\qed
\end{theorem}

In this section, we generalize Theorem~\ref{thm:prodfunc_multi}.
This is done by Theorem~\ref{thm:posdef}, where we omit the restrictive 
assumption~(\ref{bocher:assume:form}) on the kernel's translation-invariance.
To this end, we introduce specific instances of structured interpolation point sets, called {\em grid-like} data, in Section~\ref{sec:GridLike}. 
Then, in Section~\ref{sec:stability}, the numerical stability of interpolation by product kernels is analyzed.

\subsection{Interpolation on grid-like data}\label{sec:GridLike}
In order to show positive definiteness of product kernels with positive definite components, two ingredients are essential:
\begin{itemize}
\item
point sets of {\em grid-like} structure;
\item
the Kronecker product.
\end{itemize}

\begin{defi}
	If a point set $X \subset \mathbb{R}^d$ can be written as a Cartesian product
	\begin{align} \label{eq:grid_like}
	X = X^1 \times ... \times X^M = \bigtimes\limits_{i=1}^{M} X^i
	\end{align}
	between finite point sets $X^i \subset \mathbb{R}^{d_i}$, 
	where $i=1,...,M$ and $d=\sum_{i=1}^{M}d_i$, 
	then we say that $X$ has {\tt grid-like} structure, or, in short: $X$ is {\tt grid-like}.
	\qed
\end{defi}

\begin{defi}[{\cite[Definition 4.2.1]{Horn_MatrixAna}}] \label{def:kroneckerProd}
	The {\tt Kronecker product} $A \otimes B$ of two matrices $A = (a_{i,j})_{{1 \leq i \leq m \atop 1 \leq j \leq n}} \in \mathbb{R}^{m\times n}$ and $B\in \mathbb{R}^{p\times q}$ is defined as the block matrix
	\begin{align*}
	A \otimes B = \left[
	\begin{matrix}
	a_{1,1} B & \cdots & a_{1,n} B \\
	\vdots & \ddots & \vdots\\
	a_{m,1}B & \cdots & a_{m,n}B
	\end{matrix}
	\right] \in \mathbb{R}^{mp\times nq}.
	\end{align*}
	The entries of $A \otimes B$ are given by
	\begin{align*}
	\left(A\otimes B\right)_{j,k} = a_{\lceil j/p\rceil,\lceil k/q\rceil} \cdot b_{(j-1)\text{ mod }p +1,\;(k-1)\text{ mod }q +1}
	\end{align*}
	for $1 \leq j \leq mp$ and $1 \leq k \leq nq$, where $\lceil \cdot \rceil$ denotes as usual the ceiling function and $\text{`mod'}$ is the remainder after division.
	\qed
\end{defi}

The next result combines the above concepts,
whereby we show that the interpolation matrix for grid-like data can be obtained as a Kronecker product.

\begin{theorem} \label{thm:interpoMatrixAsKronecker}
	Let $K_i$ be positive semi-definite on $\mathbb{R}^{d_i}$ and $X^i \subset \mathbb{R}^{d_i}$ be finite point sets, for $i=1,\ldots,M$. 
	Moreover, let $K = \prod_{i=1}^M K_i$ and $X = \bigtimes_{i=1}^M X^i$. Then, there is an ordering of $X$, such that the corresponding interpolation matrix $A_{K,X}$ can be written as
		\begin{align*}
			A_{K,X} = \myotimes{i=1}{M} A_{K_i,X^i}.
		\end{align*}
\end{theorem}
\begin{proof} 
Let $n_i := \vert X^i \vert$ denote the number of points in $X^i$, for $i=1,\ldots,M$, and $n := \prod_{i=1}^{M} n_i$. We prove the statement by induction on $M$ as follows. 

\bigskip

	\noindent \underline{$M = 2$:} 
	For $X^1 := \left\{ x_1^1, \ldots , x_{n_1}^1 \right\}$, $X^2 := \left\{ x_1^{2} , \ldots , x_{n_2}^{2} \right\}$
	and $k \in \lbrace 1 , \ldots , n \rbrace$, we let
	\begin{align*}
		x_k = \left( x_{\lceil k/n_2 \rceil}^{1}, \;x_{(k-1) \text{ mod } n_2 +1}^{2} \right).
	\end{align*}
	This leads to an ordering $X = X^1 \times X^2 = \{ x_k \}_{k=1}^n$, which results in
	\begin{align*}
	\left(A_{K,X}\right)_{j,k}
	&=  K\left( \left( x^1_{ \lceil j/n_2 \rceil },\; x^2_{(j-1)\text{ mod }n_2 +1} \right),\;\left( x^1_{ \lceil k/n_2 \rceil }, \;x^2_{(k-1)\text{ mod }n_2 +1} \right) \right)\\
	&=  K_1 \left(x^1_{ \lceil j/n_2 \rceil },\; x^1_{ \lceil k/n_2 \rceil }  \right) \cdot \; K_2 \left( x^2_{(j-1)\text{ mod }n_2 +1},\; x^2_{(k-1) \text{ mod } n_2 +1} \right) \\
	&= \left( A_{K_1,X^1} \right)_{\lceil j/n_2 \rceil ,\; \lceil k/n_2 \rceil} \cdot \left( A_{K_2,X^2} \right)_{(j-1)\text{ mod }n_2 +1 ,\; (k-1)\text{ mod }n_2 +1}
	\end{align*}
	for $j,k \in \lbrace 1,\dots, n \rbrace$. Hence, the identity $A_{K,X} = A_{K_1,X^1} \otimes A_{K_2,X^2}$ holds.

\bigskip

	\noindent \underline{$M \longrightarrow M+1$:} 
	Let $\tilde{X} = \bigtimes_{i=1}^{M} X^i$, $\tilde{K} = \prod_{i=1}^M K_i$ and $\tilde{n} = \prod_{i=1}^M n_i$. 
	Due to the induction hypothesis, there is an ordering $\tilde{X} = \{ \tilde{x}_\ell \}_{\ell = 1}^{\tilde{n}}$ satisfying
	\begin{align*}
		A_{\tilde{K},\tilde{X}} = \myotimes{i=1}{M} A_{K_i,X^i}.
	\end{align*}
	As in the case $M = 2$, we can order $X = \tilde{X} \times X^{M+1}$ to get
	\begin{align*}
		A_{K,X} &= A_{\tilde{K}, \tilde{X}} \otimes A_{K_{M+1},X^{M+1}}.
	\end{align*}
	Thereby, we can conclude
	\begin{align*}
		A_{K,X} = \myotimes{i=1}{M+1} A_{K_i,X^i},
\end{align*}
which already completes our proof.
\end{proof}

Since the Kronecker product of positive definite matrices is positive defi\-nite (cf.~\cite[Theorem~4.2.12]{Horn_MatrixAna}), 
we can show that interpolation matrices $A_{K,X}$ for product kernels $K$ and grid-like point sets $X$ are positive definite, 
provided that {\em all} component kernels are positive definite. 
We summarize this result as follows.

\begin{corollary} \label{grid_pd}
	Under the assumptions and with using the notations of Theorem~\ref{thm:interpoMatrixAsKronecker},
	let the components $K_i$ be positive definite, and let the finite sets $X^i \subset \mathbb{R}^{d_i}$ contain pairwise distinct points, 	
	for $i=1,\ldots,M$. 
	Then, the interpolation matrix $A_{K,X}$ of the product kernel $K$ is positive definite.
	\qed
\end{corollary}

Now we are in a position, where we can show positive definiteness of the product kernel's matrix $A_{K,X}$, 
for any finite set $X$ of pairwise distinct points.

\begin{theorem} \label{thm:posdef}
	Let $K_i$ be positive definite kernels on $\mathbb{R}^{d_i}$, $i= 1,\dots, M$. 
	Then the corresponding product kernel $K$ is positive definite on $\mathbb{R}^d$, where $d = \sum_{i=1}^M d_i$.
\end{theorem}

\begin{proof}
	Let $X = \{x_1,\dots,x_n \} \subset \mathbb{R}^d$ contain pairwise distinct points. 
	Then, we find a grid-like superset
	\begin{align*}
		Y = \bigtimes_{i=1}^M  Y^i
	\end{align*}
	of $X$, i.e., $X \subset Y$, where each $Y^i \subset \mathbb{R}^{d_i}$ contains pairwise distinct points, for $i = 1,\ldots, M$ (cf.~Figure~\ref{fig:randTogrid}). 
	Due to Corollary~\ref{grid_pd}, the interpolation matrix $A_{K,Y}$ is positive definite. 
	But then, the matrix $A_{K,X}$ is a submatrix of $A_{K,Y}$, and so $A_{K,X}$ must be positive definite.
\end{proof}

\begin{figure}[h!]
	\begin{center}
		\fbox{
			\includegraphics[width=9cm]{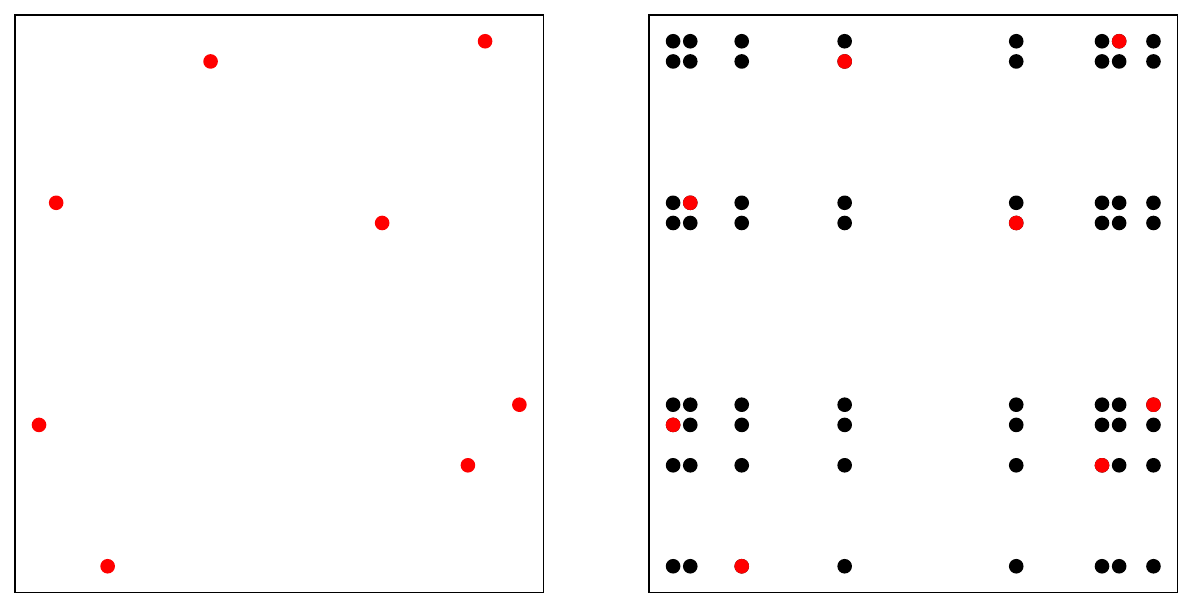}
		}
	\end{center}
	\caption{Scattered points $X$ (left) and grid-like points including $X$ (right).}
	\label{fig:randTogrid}
\end{figure}

From our previous results, we can further conclude that for tensor target functions $f$, their kernel interpolant $s_f$ is also a tensor.
We detail this as follows.

\begin{corollary} \label{tensor_interpolant}
Let $K$ be a product kernel of positive definite components 
and let $X = \{ x_1 , \ldots , x_n \} \subset \mathbb{R}^d$ be a finite set of pairwise distinct points.
Then, for any data vector $f_X = \left( f(x_1), \ldots , f(x_n) \right)^T \in \mathbb{R}^n$
of function values there is a unique interpolant $s_{f,X} \in S_{K, X}$ satisfying 
	\begin{align*}
		s_{f,X} (x) = f(x) \qquad \mbox{ for all } x \in X.
	\end{align*}
\end{corollary}

\begin{proof}
The stated assertion is a direct consequence of Theorem~\ref{thm:posdef}. 
\end{proof}

\begin{corollary} \label{tensor_interpolant2}
Let $K$ be a product kernel of positive definite components, let $f \in \mathcal{H}_{K,\Omega}$ be a target function of tensor product form
	\begin{equation}
	\label{f:tensor:product}
	    f = \prod_{i=1}^M f_i \qquad \mbox{ with } f_i \in \mathcal{H}_{K_i,\Omega_i} \quad \mbox{ for all $i = 1,\dots, M$},
	\end{equation}
and let $X = \bigtimes\limits_{i=1}^{M} X^i$ be grid-like, where 
each $X^i \subset \Omega_i \subset \mathbb{R}^{d_i}$ contains pairwise distinct points,
for $i = 1,\ldots, M$, and $\Omega := \bigtimes_{i=1}^{M} \Omega_i$.	
Then, the unique interpolant $s_{f,X} \in S_{K, X}$ to $f$ on $X$ is given by
	\begin{align*}
		s_{f,X} = \prod_{i=1}^M s_{f_i,X^i},
	\end{align*}
	where $s_{f_i,X^i}\in S_{K_i,X^i} $ is the unique interpolant to $f_i$ on the points in $X^i$.
\end{corollary}

\begin{proof}
Regard the product
	\begin{align*}
		g = \prod_{i=1}^M s_{f_i,X^i}.
	\end{align*}
Note that all factors $s_{f_i,X^i}$ are well-defined due to the assumptions on $X^i$ and $K_i$. 
As $s_{f_i,X^i} \in S_{K_i,X^i} \subset \mathcal{H}_{K_i,\Omega_i}$, we have $g \in \mathcal{H}_{K,\Omega}$, due to Theorem \ref{thm:products}. 
By elementary calculations, we find $g \in S_{K,X}$ and $g(x) = f(x)$ for all $x \in X$. 
But this implies
	\begin{align*}
		s_{f,X} = g = \prod_{i=1}^M s_{f_i,X^i},
	\end{align*}
since the interpolant $s_{f,X} $ to $f$ on $X$ is unique in $S_{K,X}$ by Corollary~\ref{tensor_interpolant}.
\end{proof}

\subsection{Stability}\label{sec:stability}
The main indicator for the stability of kernel-based reconstruction methods is the \textit{spectral condition number} of the interpolation matrix $A_{K,X}$, given by
\begin{equation}
\label{condition:number}
	\text{cond}_2 (A_{K,X}) = \frac{\sigma_{\max}(A_{K,X})}{\sigma_{\min}(A_{K,X})},
\end{equation}
where $\sigma_{\max}(A_{K,X})$ and $\sigma_{\min}(A_{K,X})$ denote the largest and smallest (positive) eigenvalues of $A_{K,X}$. 
In fact, $\text{cond}_2 (A_{K,X})$ quantifies the sensitivity of the linear system with respect to perturbations in the function values taken from $f$ on $X$.

From our previous subsection, we can further conclude that the condition number of the product interpolation matrix is given by the product of the individual condition numbers of the component interpolation matrices.

\begin{theorem} \label{product_eigenvalues}
	Let the assumptions of Theorem \ref{thm:interpoMatrixAsKronecker} hold. Then, the condition number of the corresponding product interpolation matrix is given by
	\begin{align*}
		\text{{\rm cond}}_2 (A_{K,X}) = \prod\limits_{i=1}^M \text{{\rm cond}}_2 (A_{K_i,X^i}).
	\end{align*}
\end{theorem}

\begin{proof}
	According to \cite[Theorem~4.2.12]{Horn_MatrixAna} and Theorem \ref{thm:interpoMatrixAsKronecker}, the spectrum $\sigma(A_{K,X})$ of $A_{K,X}$ can be written as
	\begin{align*}
		\sigma(A_{K,X}) = \Bigg\{ \prod_{i=1}^{M} \sigma^i \; \Bigg \vert \; \sigma^i \in \sigma(A_{K_i,X^i}) \text{ for } i = 1, ... , M \Bigg\},
	\end{align*}
	which directly implies
	\begin{align*}
		\text{cond}_2(A_{K,X}) = \frac{\sigma_{\max}(A_{K,X})}{\sigma_{\min}(A_{K,X})} =  \frac{\prod\limits_{i = 1}^{M} \sigma_{\max} (A_{K_i, X^i})}{\prod\limits_{i = 1}^{M} \sigma_{\min} (A_{K_i, X^i})} = \prod\limits_{i=1}^M \text{cond}_2 (A_{K_i,X^i}).
	\end{align*}
	This completes our proof already.         
\end{proof}

As an immediate conclusion from Theorem~\ref{product_eigenvalues} we remark that, for the sake of stability, the component kernels $K_i$ should be chosen, such that they generate sufficiently small spectral condition numbers on the component data $X^i$, for $i = 1, ... , M$. A supporting numerical result is provided in Section~\ref{numeric}.

Now we are in a position, where we can extend the result of Theorem~\ref{product_eigenvalues} to more general finite point sets $X \subset \mathbb{R}^d$.
In fact, any finite $X \subset \mathbb{R}^d$ can be completed to obtain a grid-like superset $Y \subset \mathbb{R}^d$, i.e., $X \subset Y$, 
(cf.~Figure~\ref{fig:randTogrid}). This is according to the construction in the proof of Theorem~\ref{thm:posdef}.

\begin{corollary} \label{cor:product_eigenvalues}
	For $i=1,\dots,M$, let $K_i$ be positive definite kernels on $\Omega_i \subset \mathbb{R}^{d_i}$ and let $K$ 
	be their product kernel on $\Omega = \bigtimes_{i=1}^M \Omega_i \subset \mathbb{R}^d$, where $d = \sum_{i=1}^M d_i$. 
	Moreover, let $X \subset \bigtimes_{i=1}^M \Omega_i$ contain finitely many pairwise distinct data points, and
	let $Y = Y^1\times \dots\times Y^M \in \bigtimes_{i=1}^M \Omega_i$ be a grid-like superset of $X$, i.e., $X \subset Y$.
	Then, the following statements hold.
\begin{enumerate}
		\item[(a)] $\text{\normalfont cond}_2 ({A}_{K,Y}) = \prod\limits_{i=1}^M \text{\normalfont cond}_2 ({A}_{K_i,Y^i}).$ 
		\item[(b)] $\text{\normalfont cond}_2 ({A}_{K,X}) \leq \text{\normalfont cond}_2 ({A}_{K,Y}).$ 
		\item[(c)] $\lambda_{\min} ({A}_{K,X}) \geq \prod\limits_{i=1}^M \lambda_{\min} ({A}_{K_i,Y^i}).$ 
\end{enumerate}
\end{corollary}

\begin{proof}
Statement~(a) is covered by Theorem~\ref{product_eigenvalues}.
Statement~(b) follows directly from the Cauchy interlacing theorem, since ${A}_{K,X}$ is a principal submatrix of ${A}_{K,Y}$, whereby we have
	\begin{equation*}
		\lambda_{\min}({A}_{K,Y}) \leq \lambda_{\min}({A}_{K,X}) 
		\qquad \text{ and } \qquad
		\lambda_{\max}({A}_{K,X}) \leq \lambda_{\max}({A}_{K,Y}).
	\end{equation*}  

But this, in combination with Theorem~\ref{product_eigenvalues}, implies statement~(c).
\end{proof}

For commonly used kernel functions, estimates on the eigenvalues of their interpolation matrices can be found in the literature, 
see e.g.~\cite[Theorem~12.3]{wendland}. 
These can obviously be combined with Theorem \ref{product_eigenvalues} to derive condition number estimates for interpolation matrices of product kernels.

\section{Tensor Newton basis} \label{Newton_basis}
In \cite{Muller2009article}, a Newton basis for kernel-based interpolation was introduced to achieve a stable and efficient reconstruction method. The construction of this orthonormal basis is based on the Cholesky decomposition of the interpolation matrix $A_{K,X}$. We use the properties of product kernels to reduce the computational costs for computing the Newton basis significantly. This is particularly useful for update strategies and the implementation of greedy algorithms, see Section~\ref{sec:greedy}.
On this occasion and in view of possible extensions of our theory to conditionally positive definite kernels, 
we remark that full-rank orthonormal bases for such kernels were recently made available~\cite{Mohammadi2024}.
Nevertheless, we restrict ourselves to positive (semi-)definite kernels.

To this end, we assume the conditions of Theorem \ref{thm:interpoMatrixAsKronecker} to hold, such that the interpolation matrix is given by the Kronecker product
\begin{align*}
    A_{K,X} = \myotimes{i=1}{M} A_{K_i,X^i}.
\end{align*}
The Cholesky factor $L$ of $A_{K,X}$ is given by the Kronecker product
\begin{align} \label{cholesky_factor}
    L = \myotimes{i=1}{M} L_i,
\end{align}
where $L_i$ denotes the Cholesky factor of $A_{K_i,X^i}$ for $i = 1, ... , M$. Note that the Cholesky factor coincides with the Vandermonde matrix of the Newton basis. This identity indicates that we can combine the individual Newton bases in the component spaces to compute the Newton basis of the product kernel.

Due to the structure of the point set, we can write the approximation space $S_{K, X}$ as the Hilbert tensor product
\begin{align*}
    S_{K,X} = \myotimes{i=1}{M} S_{K_i,X^i}.
\end{align*}
If $\mathcal{N}_i$ denotes the Newton basis of $S_{K_i,X^i}$, then Theorem \ref{thm:hilbert_tensor}~(b) in combination with Theorem \ref{thm:products} proves that
\begin{align} \label{tensor_basis}
    \mathcal{N} := \Big\lbrace \prod_{i=1}^M \mathfrak{n}_i \; \Big\vert \; (\mathfrak{n}_1,...,\mathfrak{n}_M) \in \mathcal{N}_1 \times ... \times \mathcal{N}_M \Big\rbrace
\end{align}
is an orthonormal basis of $S_{K,X}$. 
Then, the ordering in Theorem \ref{thm:interpoMatrixAsKronecker} on the Vandermonde matrix of $\mathcal{N}$ with respect to $X$ yields the Cholesky factor (\ref{cholesky_factor}), so that these two approaches result in the same basis. We refer to this basis as the \textit{tensor Newton basis} of $S_{K,X}$.

For an efficient implementation of the tensor Newton basis, we use the update formulae from \cite{Pazouki2011} to compute the evaluation matrices
\begin{align*}
    V_i = \left( \mathfrak{n}(x^i) \right)_{\substack{x^i \in X^i \\ \mathfrak{n} \in \mathcal{N}_i}},
\end{align*}
where $\Omega_i \subset \mathbb{R}^{d_i}$ is the $i$-th component of the initial domain. In this case, the evaluation matrix of the tensor Newton basis at the point set $\Omega = \bigtimes_{i=1}^M \Omega_i$ is again given by the Kronecker product of the individual Vandermonde matrices
\begin{align*}
    V = \left( \mathfrak{n} (x) \right) _{\substack{x \in X \\ \mathfrak{n} \in \mathcal{N}}} = \myotimes{i=1}{M} V_i.
\end{align*}
In order to determine the coefficients of the interpolant with respect to the tensor Newton basis, 
we need to solve the triangular linear system, whose matrix is the Cholesky factor in (\ref{cholesky_factor}).
We point out, that one has to take particular care of the right ordering of the involved arrays when implementing this algorithm.

\section{Componentwise point selection} \label{sec:greedy}
In kernel-based interpolation, the geometric distribution of the data points is critical to both the approximation quality and the numerical stability of the reconstruction.
This has motivated the development of \textit{greedy point selection methods}, dating back to the construction of efficient \textit{thinning algorithms}~\cite{Floater_AI:1996} 
that were later refined in~\cite{Floater_AI:1998}, and optimized in~\cite{Iske_Habil}. 
For more recent contributions to greedy algorithms for kernel interpolation, 
especially for their analysis, we refer to~\cite{santin2023optimality,wenzel2021novel,Wenzel2023standard}. 

To briefly explain greedy point selection methods, these are recursive point removal (or insertion) schemes that require, at each removal step,
efficient computations of maximizers of kernel-specific error functionals, along with efficient update strategies.
Since the above mentioned greedy algorithms~\cite{DeMarchi2005,Floater_AI:1996,Floater_AI:1998,Iske_Habil,Wenzel2023standard} 
work for general positive definite kernel functions, they can in particular be applied to positive definite product kernels.

However, on given initial grid-like point set, the reduced point set (output by a greedy algorithm) will in general not be grid-like anymore.
Therefore, we cannot immediately rely on the findings of Section \ref{interpolation}, where computational advantages on grid-like data were explained. 

Therefore, we now propose more sophisticated point selection strategies that act on the product kernel components $K_i$ and maintain the grid-like structure of the interpolation point set, where we restrict the componentwise selection to subsets $\Omega_i \subset \mathbb{R}^{d_i}$ for $i = 1, ..., M$. 
This leads us to nested sequences $\left( X_n \right)_{n \in \mathbb{N}}$ in $\Omega = \bigtimes_{i=1}^M \Omega_i \subset \mathbb{R}^d$ that have the form
\begin{align*}
  X_n = X^1_n \times ... \times X^M_n \qquad \text{ where } X^i_n \subset \Omega_i \; \text{ for } i=1, ... , M, \; n \in \mathbb{N}.
\end{align*}

To preserve the grid-like structure, we will successively determine an update index $i_n \in \lbrace 1, ... , M \rbrace$ and extend the respective component point set $X_n^{i_n}$ with an element $x^{i_n} \in \Omega_{i_n} \setminus X_n^{i_n}$. Thereby, the next point sets are given by
\begin{align} \label{component_update}
  X^{i}_{n+1} =
  \begin{cases}
    X^{i}_n, & \; \text{if } i \neq i_n \\[1.5ex]
    X_n^{i_n} \cup \lbrace x^{i_n} \rbrace ,& \; \text{if } i = i_n
  \end{cases} \quad \text{and} \quad X_{n+1} = X_{n+1}^1 \times ... \times X_{n+1}^M.
\end{align}
According to our discussion in Section~\ref{Newton_basis}, we only have to update the Newton basis $\mathcal{N}_{i_n}$ of the component space $S_{K_i, X^{i_n}_n}$ to update the Newton basis of the product kernel. If $\mathfrak{n}_{i_n} \in S_{K_i, X^{i_n}_{n+1}}$ is the basis element that updates $\mathcal{N}_{i_n}$, the functions
\begin{align} \label{basis_update}
  \mathcal{N}_{{\rm new}} = 
  \Big\lbrace \mathfrak{n}_{i_n} \cdot \prod\limits_{ i \neq i_n} \mathfrak{n}_i \; \Big\vert \; 
  \left( \mathfrak{n}_1, ... , \mathfrak{n}_{i_n - 1}, \mathfrak{n}_{i_n + 1}, ... , \mathfrak{n}_M \right) \in \bigtimes\limits_{i \neq i_n} \mathcal{N}_i \Big\rbrace
\end{align}
update the Newton basis $\mathcal{N}$ of the product kernel. Given a fixed target function $f \in \mathcal{H}_{K,\Omega}$, the new interpolant $s_{f,X_{n+1}}$ can be written as
\begin{align*}
  s_{f,X_{n+1}} = s_{f,X_n} + \sum\limits_{\mathfrak{n} \in \mathcal{N}_{{\rm new}}} c_{\mathfrak{n}} \cdot \mathfrak{n},
\end{align*}
where the coefficients $c_{\mathfrak{n}} \in \mathbb{R}$ of the new basis elements have to be computed. Similar to the single point strategies, we can derive an implicit formula for the missing coefficients. To this end, consider the update point set
\begin{align*}
  X_{{\rm new}} := X^1_n \times ... \times X^{i_n - 1}_n \times \lbrace x^{i_n} \rbrace \times X^{i_n + 1}_n \times ... \times X^M_n.
\end{align*}
For every point $x \in X_{{\rm new}}$, we have 
\begin{align*}
  f(x) = s_{f,X_{n+1}} (x) = s_{f,X_n}(x) + \sum\limits_{\mathfrak{n} \in \mathcal{N}_{{\rm new}}} c_{\mathfrak{n}} \cdot \mathfrak{n}(x),
\end{align*}
or equivalently,
\begin{align*}
  f(x) - s_{f,X_n} (x) = \sum\limits_{\mathfrak{n} \in \mathcal{N}_{{\rm new}}} c_{\mathfrak{n}} \cdot \mathfrak{n}(x).
\end{align*}
Thus, the missing coefficients are given by the solution of the linear system
\begin{align} \label{update_system}
  V_{\mathcal{N}_{{\rm new}}, X_{{\rm new}}} \cdot c = r_{X_{{\rm new}}},
\end{align}
where the involved arrays are given by
\begin{align*}
  V_{\mathcal{N}_{{\rm new}}, X_{{\rm new}}} = \left( \mathfrak{n}(x) \right)_{\substack{x \in X_{{\rm new}} \\ \mathfrak{n} \in \mathcal{N}_{{\rm new}}}} 
  \qquad \text{and} \qquad
  r_{X_{{\rm new}}} = \left( f(x) - s_{f,X_n} (x) \right)_{x \in X_{{\rm new}}}.
\end{align*}
Due to the structure of $\mathcal{N}_{{\rm new}}$ (cf.~(\ref{basis_update})), we can write
\begin{align*}
  V_{\mathcal{N}_{{\rm new}}, X_{{\rm new}}} = \mathfrak{n}_{i_n} (x^{i_n}) \cdot \bigotimes\limits_{i \neq i_{n}} L_i,
\end{align*}
where $L_i$, for $i = 1, ... , M$, are the Cholesky factors in~(\ref{cholesky_factor}).

For the componentwise point selection, we will focus on the \textit{power function}
\begin{align*}
  P_{X_n} (x) := \Vert K(\cdot , x) - s_{K(\cdot,x),X_n} \Vert_K \qquad \text{ for } x \in \Omega.
\end{align*}
We consider the power functions with respect to the component kernels, i.e.,
\begin{align*}
  P_{X_n^i} (x^i) := \Vert K_i(\cdot , x^i) - s_{K_i(\cdot,x^i),X_n^i} \Vert_{K_i} \qquad \text{for } x^i \in \Omega_i \mbox{ and } \; i= 1, ... , M.
\end{align*}
Recall that $\mathfrak{n}_{i_n} (x^{i_n}) = P_{X_n^{i_n}} (x^{i_n})$ holds~\cite{Pazouki2011}, 
so that~(\ref{update_system}) can be written as the triangular system
\begin{align} \label{triangular_updateSystem}
  \bigotimes\limits_{i \neq i_{n}} L_i \cdot c = P_{X^{i_n}_{n}} (x^{i_n})^{-1} \cdot r_{X_{{\rm new}}}.
\end{align}
Note that we divide by $P_{X^{i_n}} (x^{i_n})$ in~(\ref{triangular_updateSystem}). Likewise, we need to divide by $P_{X^{i_n}} (x^{i_n})$, when performing the Newton basis update in~\cite{Pazouki2011}. Hence, to improve the numerical stability, it is reasonable to maximize the power function value in each iteration step. As a single point update, this algorithm is known as the \textit{P-greedy algorithm}~\cite{DeMarchi2005}. 
We introduce its componentwise version here, which we will refer to as the \textit{componentwise P-greedy algorithm}. 

For an update index $i_n \in \lbrace 1, ... , M \rbrace$ satisfying
\begin{align} \label{Pgreedy_index}
  \underset{y^{i_n} \in \Omega_{i_n}}{\sup} \; P_{X_n^{i_n}}(y^{i_n}) =: \Vert P_{X_n^{i_n}} \Vert_{\infty, \Omega_{i_n}} = \underset{i =1, ... , M}{\max} \; \Vert P_{X_n^{i}} \Vert_{\infty, \Omega_i}
\end{align}
we select the update point from $\Omega_{i_n}$ via
\begin{align} \label{Pgreedy_update}
  P_{X_n^{i_n}} (x^{i_n}) = \underset{y^{i_n} \in \Omega_{i_n}}{\sup} \; P_{X_n^{i_n}}(y^{i_n}).
\end{align}
Then, the updated point set $X_{n+1}$ is given by~(\ref{component_update}). 
To prove convergence for the componentwise P-greedy algorithm, we need the following two lemmas.

\begin{lemma} \label{power_product}
  Let $K = \prod_{i=1}^{M} K_i$ denote a product kernel with positive definite components and 
  $X = X^1 \times \cdots \times X^M \subset \mathbb{R}^d$ be a grid-like point set. 
  Then, the power function can be written as
  \begin{align*}
    \left( P_{X} (x) \right)^2 = \prod_{i=1}^{M} K_i(x^i, x^i) - \prod_{i=1}^{M} \left( K_i(x^i, x^i) - \left( P_{X^i} (x^i) \right)^2 \right) \qquad \text{for } x \in \mathbb{R}^d,
  \end{align*}
  where $x^i$ denotes the $i$-th component of $x \in \mathbb{R}^d$ for $i = 1, ... , M$.
\end{lemma}

\begin{proof}
  Let $\mathcal{N}$ denote the Newton basis of $S_{K,X}$ and $\mathcal{N}_i$ denote the component Newton bases of $S_{K_i,X^i}$, for $i = 1, ... , M$. 
  Due to orthonormality, we have
    $$
    \left( P_X(x) \right)^2 = K(x,x) - \sum\limits_{\mathfrak{n} \in \mathcal{N}} \left( \mathfrak{n} (x) \right)^2 
    $$
and
    $$
    \left( P_{X^i}(x^i) \right)^2 = K_i(x^i,x^i) - \sum\limits_{\mathfrak{n}_i \in \mathcal{N}_i} \left( \mathfrak{n}_i (x^i) \right)^2
    $$
  for $x \in \mathbb{R}^d, x^i \in \mathbb{R}^{d_i}$ and $i = 1, ... , M$. This, in combination with~(\ref{tensor_basis}) yields
  \begin{align*}
    \left( P_X(x) \right)^2 &= K(x,x) - \sum\limits_{\mathfrak{n} \in \mathcal{N}} \left( \mathfrak{n} (x) \right)^2 \\
    &= \prod_{i=1}^{M} K_i(x^i, x^i) - \sum_{\mathfrak{n}_1 \in \mathcal{N}} \dots \sum_{\mathfrak{n}_M \in \mathcal{N}_M} \; \prod_{i=1}^{M} \left( \mathfrak{n}_i (x^i) \right)^2 \\
    &= \prod_{i=1}^{M} K_i(x^i, x^i) - \prod_{i=1}^{M} \left( \sum_{\mathfrak{n}_i \in \mathcal{N}_i} \left( \mathfrak{n}_i(x^i) \right)^2 \right) \\
    &= \prod_{i=1}^{M} K_i(x^i, x^i) - \prod_{i=1}^{M} \left( K_i(x^i, x^i) - \left( P_{X^i} (x^i) \right)^2 \right)
  \end{align*}
  for all $x \in \mathbb{R}^d$.
\end{proof}

\begin{lemma}
\label{new-lemma}
    Let $K: \mathbb{R}^d \times \mathbb{R}^d \longrightarrow \mathbb{R}$ be positive definite and $\Omega \subset \mathbb{R}^d$. Then, the following statements hold:
    \begin{itemize}
        \item[(a)] If $\left(X_n \right)_{n \in \mathbb{N}}$ is a sequence of subsets of $\Omega$, such that the power function decays pointwise to zero on $\Omega$, i.e.
        \begin{align*}
            P_{X_n} (x) \xrightarrow{n \to \infty} 0 \qquad \text{for all } x \in \Omega,
        \end{align*}
        then we have normwise convergence for the interpolation method, i.e.,
        \begin{align*}
            \Vert f - s_{f,X_n} \Vert_K \xrightarrow{n \to \infty} 0 \qquad \text{for all } f \in \mathcal{H}_{K,\Omega}.
        \end{align*}
        \item[(b)] 
        Suppose $K$ is continuous and $\Omega \subset \mathbb{R}^d$ is compact. Then, we have
        \begin{align*}
            P_{X_n} (x_{n+1}) \xrightarrow{n \to \infty} 0
        \end{align*}
        for any sequence $\left( x_n \right)_{n \in \mathbb{N}}$ in $\Omega$, where $X_n = \left\{ x_1, \ldots , x_n \right\}$ for all $n \in \mathbb{N}$.
    \end{itemize}
\end{lemma}

\begin{proof}
Statement~(a) is covered by Karvonen's result~\cite[Proposition~2.1]{Karvonen2022}.

\medskip
Statement~(b) is a special case of~\cite[Lemma~4.2]{Albrecht_AI:2023}.
Nevertheless, let us adapt our proof from \cite{Albrecht_AI:2023} to the present case of Lagrange interpolation.

\smallskip
 To prove (b), let $\varepsilon > 0$. Due to our assumptions on $K$ and $\Omega$, the mapping $x \longmapsto K(\cdot, x)$ is uniformly continuous on $\Omega$, so that we find $\delta > 0$ satisfying
    \begin{align*}
        \Vert K(\cdot, y) - K (\cdot, x) \Vert_{K} < \varepsilon.
        \qquad 
        \text{ for all $x,y \in \Omega$ with } \Vert x - y \Vert < \delta.
    \end{align*} 
    Moreover, there are finitely many $y_j \in \mathbb{R}^d$, $j = 1, ... , L$, for $L \in \mathbb{N}$, such that
    \begin{align*}
        \Omega \subset \bigcup\limits_{j=1}^L B_{\delta / 2} (y_j). 
    \end{align*}
    Given this covering of $\Omega$, we let
    \begin{align*}
        I_j = \lbrace n \in \mathbb{N} \mid x_n \in B_{\delta/2} (y_j) \rbrace 
        \qquad \text{ and } \qquad 
        N_j = 
        \left\{
        \begin{array}{cl} 
            \min (I_j) & \text{ if } I_j \neq \emptyset \\[1.5ex]
            0 &\text{ if } I_j = \emptyset
        \end{array}
        \right.
    \end{align*}
    for $j = 1, ..., L$, and choose $N = \max (N_1, ... , N_L)$. 
    Then, for every $n \geq N$, there is at least one pair $(j,x)$ of an index $1 \leq j \leq L$ and a point $x \in X_N \subset X_n$, 
    such that $x_{n+1}, x \in B_{\delta/2}(y_{j})$, whereby
    \begin{align*}
        \Vert x_{n+1} - x \Vert \leq \Vert x_{n+1} - y_j \Vert + \Vert y_j - x \Vert < \delta,
    \end{align*}
    and so
    \begin{align*}
        P_{X_n} (x_{n+1}) \leq \Vert K(\cdot, x_{n+1}) - K (\cdot, x) \Vert_K < \varepsilon.
    \end{align*}
\end{proof}

Now we are finally in a position, where we can prove convergence of the tensor product interpolation method,
in situations where the power functions of the component kernels decay to zero.
But this is the case for the componentwise P-greedy algorithm, if the kernel $K$ is continuous and the domain $\Omega$ is compact.
We remark that our following result requires only very mild conditions on the kernel $K$ and the target $f$.
This is in contrast to \cite{Santin2017,Wenzel2023standard}, where convergence rates were proven under more restrictive conditions on $K$ and $f$.

\begin{theorem}
  Let $K = \prod_{i=1}^{M} K_i$ be a product kernel with positive definite components and $\Omega = \bigtimes_{i=1}^M \Omega_i \subset \mathbb{R}^d$, where $\Omega_i \subset \mathbb{R}^{d_i}$.
  \begin{itemize}
    \item[(a)] 
    If $\left( X_n \right)_{n \in \mathbb{N}}$ is a sequence of grid-like subsets in $\Omega$, i.e.,
    \begin{align*}
      X_n = X^1_n \times ... \times X^M_n \qquad \text{ where } X^i_{n} \subset \Omega_i \mbox{ for } i = 1, \ldots , M,
    \end{align*}
    satisfying
    \begin{align*}
      P_{X_n^i} (x^i) \xrightarrow{n \to \infty} 0 \qquad \text{for } x^i \in \Omega_i, \; i = 1, ... , M, 
    \end{align*}
    then we have the convergence
    \begin{align*}
      \Vert f - s_{f,X_n} \Vert_K \xrightarrow{n \to \infty} 0 \qquad \text{for all } f \in \mathcal{H}_{K,\Omega}.
    \end{align*}
    \item[(b)] 
    Let $K_i$ be continuous and $\Omega_i$ be compact for $i=1, ... , M$. 
    If $\left( X_n \right)_{n \in \mathbb{N}}$ is chosen via the componentwise P-greedy algorithms from (\ref{Pgreedy_index}) and (\ref{Pgreedy_update}), 
    then we have the convergence
    \begin{align*}
      \Vert f - s_{f,X_n} \Vert_K \xrightarrow{n \to \infty} 0 \qquad \text{for all } f \in \mathcal{H}_{K,\Omega}.
    \end{align*}
  \end{itemize}
\end{theorem}

\begin{proof}
To prove statement~(a), we use the representation from Lemma~\ref{power_product},
according to which the decay of the component power functions leads to
\begin{align*}
    P_{X_{n}}(x) \xrightarrow{n \to \infty} 0 \qquad \text{ for all } x \in \Omega. 
\end{align*}
Therefore, the stated convergence follows directly from~Lemma~\ref{new-lemma}~(a).

To prove statement~(b), let $i_n$, for $n \in \mathbb{N}$,  be the update index from~(\ref{Pgreedy_index}).
We let
\begin{align*}
   N_i := \vert \lbrace n \in \mathbb{N} \mid i_n = i \rbrace \vert \in {\mathbb N} \cup \{\infty\} \qquad \text{ for } i = 1, ... , M.
\end{align*}
Suppose $N_i = \infty$. Then, the convergence $\Vert P_{X_n^i} \Vert_{\infty, \Omega_i} \longrightarrow 0$, for $n \to \infty$,   
follows from~Lemma~\ref{new-lemma}~(b), due to the definition of the selection rule, cf.~(\ref{Pgreedy_update}).
However, since the componentwise algorithm always chooses the update index with the largest $L^\infty$ norm 
of the power function, $N_i$ cannot be finite. This implies
\begin{align*}
     \Vert P_{X_n^i} \Vert_{\infty, \Omega_i} \xrightarrow{n \to \infty} 0 \qquad \text{for } i=1,...,M,
\end{align*}
so that the assertion follows from statement~(a).
\end{proof}
 
\begin{remark}
If the target function is a tensor product of the form~(\ref{f:tensor:product}), 
then its interpolant is also a tensor product, according to Corollary~\ref{tensor_interpolant2}.
Therefore, one could be tempted to transfer target-dependent algorithms like the \textit{f-greedy algorithm}~\cite{Schaback2000}
or the \textit{psr-greedy algorithm}~\cite{Dutta2021} to the single components. 
But the structure of the target function is usually unknown,
and so we cannot assume that the target function is a tensor product of the form~(\ref{f:tensor:product}).
\end{remark}

\section{Numerical results}\label{numeric}
For the purpose of illustration, we work with a two-dimensional numerical example, i.e., $d=2$, although the proposed method of this paper works for arbitrary dimensions $d$. 
Nevertheless, the following numerical experiments support our theoretical findings along with the effects that we have described.

Recall that a (bivariate) product kernel $K : \mathbb{R}^2 \times \mathbb{R}^2 \longrightarrow \mathbb{R}$ has the form
$$
    K(x,y) = K_1(x_1,y_1) \cdot K_2(x_2,y_2)
    \qquad \mbox{ for } x,y \in \mathbb{R}.
$$
To obtain $K$, we work with two types of radial compactly supported  component kernels 
$K_i : \mathbb{R} \times \mathbb{R} \longrightarrow \mathbb{R}$, $i=1,2$.
One is {\it Askey's radial charac\-teristic function} (cf.~\cite[Example~8.11]{Iske_Approx}),
$$
    \phi_{\beta}(r) = (1-r)_+^\beta
     \qquad \mbox{ for } r \geq 0,
$$
with parameter $\beta \geq 2$, the other is the $C^6$ {\em Wendland kernel} for dimension $d=1$ (computed from \cite[Corollary~9.14]{wendland}),
$$
    \phi_{1,3}(r) = (1 - r)^7_+  ( 315 r^3 + 285 r^2 + 105 r + 15).
    \qquad \mbox{ for } r \geq 0.
$$
The product of these two kernels yields the product kernel $K = \phi_\beta \cdot \phi_{1,3}$.

Further in our numerical experiments, we compare the (bivariate) pro\-duct kernel $K$ 
with Askey's kernel $\phi_8$ and Wendland's $C^6$ kernel (cf.~\cite[Table~9.1]{wendland})
$$
    \phi_{3,3}(r) = (1-r)_+^8 (32 r^3 + 25 r^2 +8r + 1)
    \qquad \mbox{ for } r \geq 0,
$$
each taken as a bivariate function. 
The kernels $\phi_{3,3}$, $\phi_8$ and $K$ are shown in Figure~\ref{fig:supportKernels}.

\begin{figure}[h!]
	\begin{center}
		\fbox{
			\includegraphics[width=11cm]{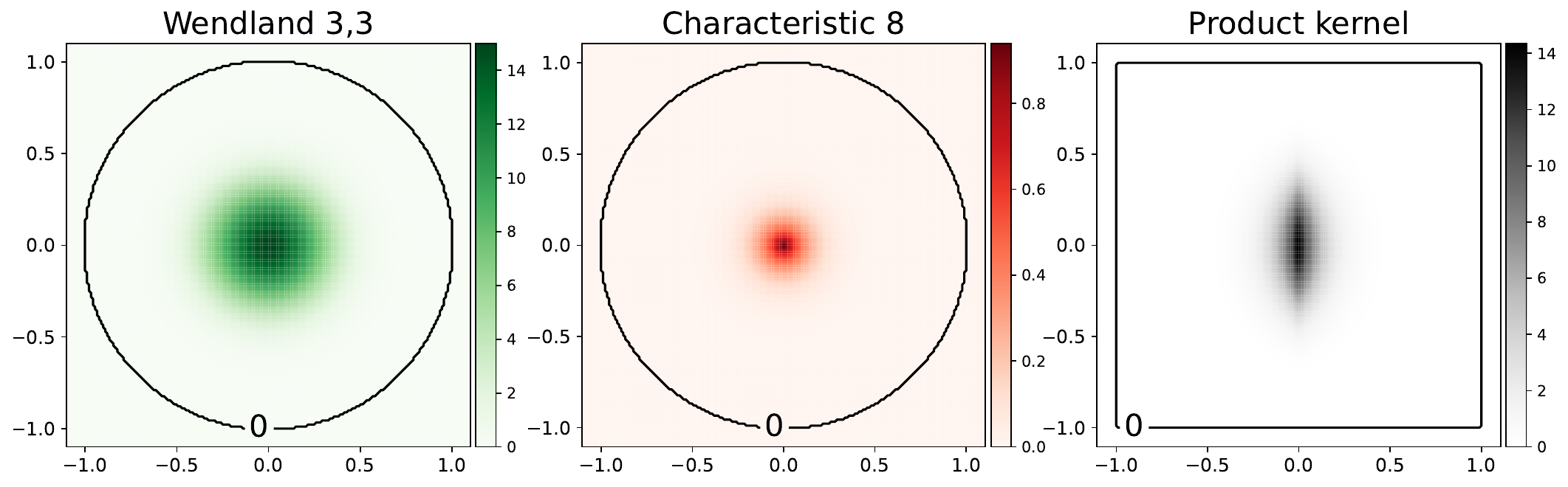}
		}
	\end{center}
	\caption{Kernels $\phi_{3,3}$, $\phi_8$, $K$ (left to right), and their support (black line).
	}
	\label{fig:supportKernels}
\end{figure}

Further in our setting, we work with grid-like sets of interpolation points ${X=X^1\times X^2} \subset {\mathbb R^2}$, 
where the mesh width of the equidistant points in $X^2 \subset {\mathbb R}$ are smaller than those of $X^1 \subset {\mathbb R}$,
see Figure~\ref{fig:targetF+Nodes} for one example. 
More precisely, in our numerical experiments we considered using data point sets $X_{i,j} = X_i \times X_j$, 
for $1\leq i \leq 4$ and $1\leq j \leq 8$, where
$$
    {X_j = \{2^{-j}k \mid k = 0,\dots,2^j\}\subset \mathbb{R}},
$$
and so we have ${|X_j| = 2^j +1}$ for the cardinalities of the component sets $X_j$, i.e.,
\begin{equation}
\label{component:set}
    \{ |X_j| \mbox{ {\bf :} } 1\leq j \leq 8 \} = \{ 3, 5, 9, 17, 33, 65, 129, 257 \}.
\end{equation}

\begin{figure}[h!]
	\begin{center}
		\fbox{
			\includegraphics[width=8cm]{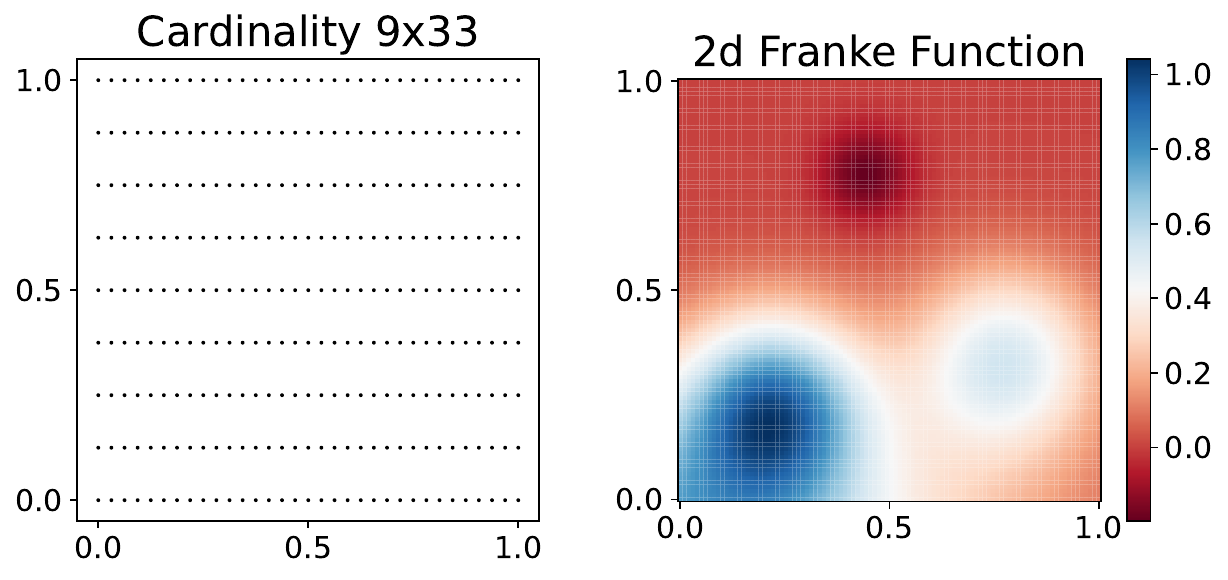}
		}
	\end{center}
	\caption{Grid-like data $X_{3,5} = X_3 \times X_5$ (left) and Franke's function (right).
	}
	\label{fig:targetF+Nodes}
\end{figure}

For interpolation on the sets $X_{i,j} = X_i \times X_j$, we used
{\em Franke's function} 
\begin{eqnarray*}
\lefteqn{ 
f(x,y):=
}\\ && 
0.75 \exp \left( - \frac{(9x-2)^2 + (9y - 2)^2}{4} \right)
+\; 0.75 \exp \left( - \frac{(9x+1)^2}{49} - \frac{(9y + 1)^2}{10}\right)
\\ &&
+\; 0.5 \exp\left( - \frac{(9x - 7)^2 + (9y - 3)^2}{4}\right) 
-\; 0.2 \exp\left( -  (9x - 4)^2 - (9y - 7)^2  \right),
\end{eqnarray*}
as target. 
But we also used its restriction $f(\cdot ,0.25) : {\mathbb R} \longrightarrow {\mathbb R}$ for experiments concerning univariate interpolation.  
We remark that Franke's function~\cite{franke1982scattered}, 
shown in Figure~\ref{fig:targetF+Nodes}, is a popular test case in scattered data approximation.

\subsection{Numerical stability versus approximation behaviour}

Now let us turn to our numerical comparisons,
where we start with numerical stability.
We recorded the spectral condition numbers ${\rm cond}_2(A_{K,X})$ in (\ref{condition:number}),
for the kernels $\phi_{1,3}$ and $\phi_8$, each on the eight component sets $X_j$ in (\ref{component:set}), for $j=1,\ldots,8$.
The spectral condition numbers ${\rm cond}_2(A_{K_i,X_j})$ of the kernel matrices $A_{K_i,X_j}$, for $i=1,2$ and $j=1,\ldots,8$, 
are shown in Figure~\ref{fig:1dCondError} (left).

\begin{figure}[h!]
	\begin{center}
		\fbox{
			\includegraphics[width=12cm]{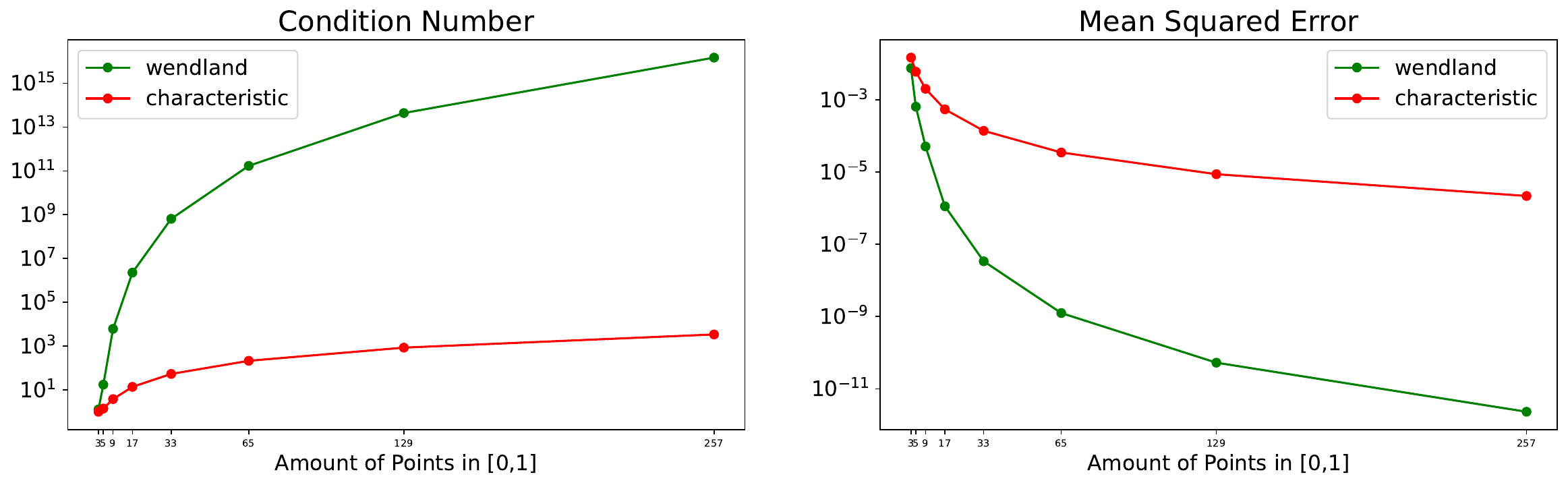}
		}
	\end{center}
	\caption{Comparison between kernels $\phi_{1,3}$ and $\phi_8$.
	Condition number of kernel matrices on $X_j$ in (\ref{component:set}), $j=1,\ldots,8$ (left); mean square error (right).}
	\label{fig:1dCondError}
\end{figure}

As regards their approximation behaviour, we have, for both test cases, also recorded the resulting {\em mean square error (MSE)},
as shown in Figure~\ref{fig:1dCondError} (right). 

Note that the observed behaviour in Figure~\ref{fig:1dCondError}
complies with Schaback's uncertainty relation~\cite{schaback1995error}, i.e., the Askey kernel $\phi_8$
leads, in comparison with the Wendland kernel $\phi_{1,3}$, to smaller condition numbers, but the mean square errors of $\phi_8$ are,
on each set $X_j$, $j=1,\ldots,8$, larger than those of $\phi_{1,3}$.

\medskip
Next we demonstrate how we can combine the two kernels $\phi_8$ and $\phi_{1,3}$ to obtain a problem-adapted product kernel for bivariate interpolation. 
We continue to work with the component sets $X_j$ in (\ref{component:set}) to build grid-like point sets of the form $X = X^1 \times X^2$, as already explained in the outset of this section.

We take the bivariate product kernel 
$$
    K(x,y) = K_1(x_1,y_1) \cdot K_2(x_2,y_2)
    \qquad \mbox{ for } x = (x_1,x_2), y = (y_1,y_2) \in {\mathbb R}^2,
$$
where we let $K_1(x_1,y_1) = \phi_8(|x_1 - y_1|)$ and $K_2(x_2,y_2) = \phi_{1,3}(|x_2 - y_2|)$. 
This particular choice is well-motivated by the result of Theorem~\ref{product_eigenvalues}, 
especially since numerical stability is a critical issue, cf.~Figure~\ref{fig:1dCondError} (left).

\begin{figure}[h!]
	\begin{center}
		\fbox{
			\includegraphics[width=11cm]{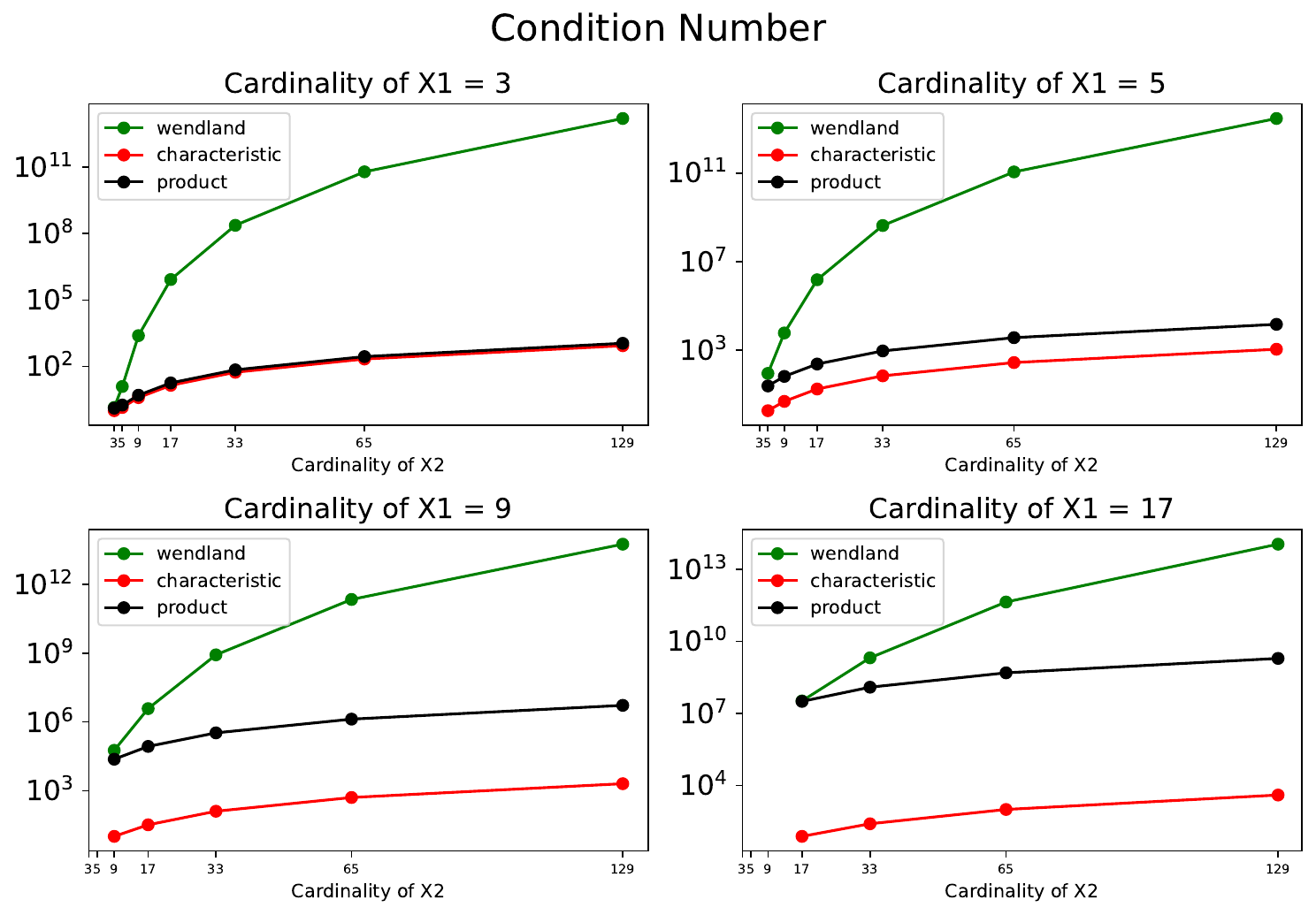}
		}
	\end{center}
	\caption{Condition numbers from kernels $K_1$, $K_2$ and $K$ on $X^1 \times X_2$.
	}
	\label{fig:CompCondition}
\end{figure}

Figure~\ref{fig:CompCondition} shows the spectral condition number for the product kernel $K$ on a sequence of grid-like point sets $X = X^1 \times X^2$.
In comparison, Figure~\ref{fig:CompCondition} also shows the condition number for both kernels $\phi_8$ and $\phi_{3,3}$, each taken as 
a {\em bivariate} radial kernel.
As can be seen in Figure~\ref{fig:CompCondition}, the condition numbers of the product kernel increase at about the same rate as that of the Askey kernel $\phi_8$.

For all test cases, we have also recorded the resulting mean square error.
Our numerical results are shown in Figure~\ref{fig:CompError}.
Note that the approximation behaviour of the product kernel $K$ is superior to that of the Askey kernel $\phi_8$.
Moreover, the mean square errors from the product kernel $K$ decrease, 
at increasing cardinality $|X_2|$, towards the mean square error resulting from the Wendland kernel $\phi_{3,3}$.

\begin{figure}[h!]
	\begin{center}
		\fbox{
			\includegraphics[width=11cm]{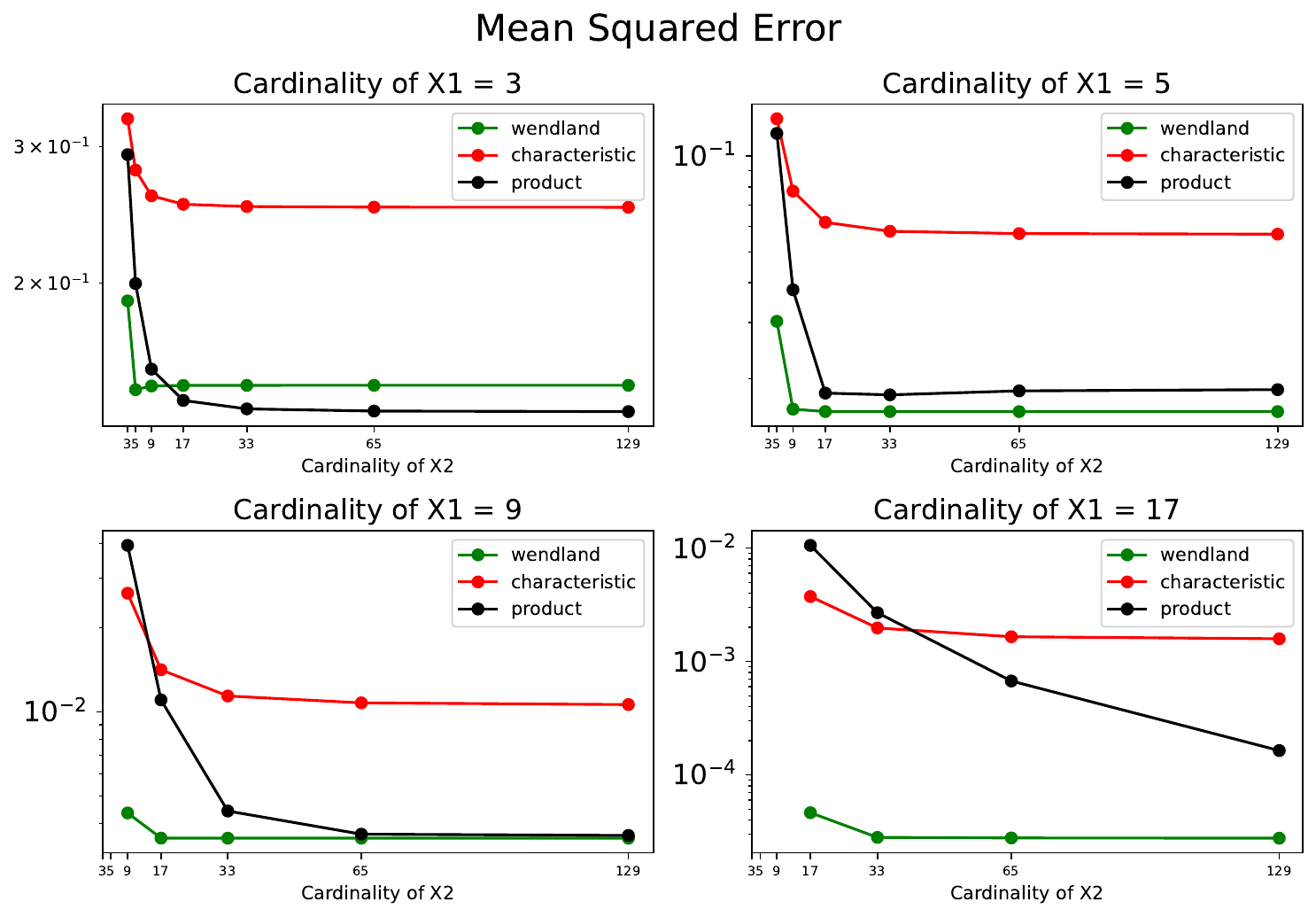}
		}
	\end{center}
	\caption{Mean square errors from kernels $K_1$, $K_2$ and $K$ on $X^1 \times X_2$.}
	\label{fig:CompError}
\end{figure}

From our numerical results in Figures~\ref{fig:CompCondition} and~\ref{fig:CompError}, 
we can conclude that the proposed product kernel $K(x,y) = K_1(x_1,y_1) \cdot K_2(x_2,y_2)$ achieves to combine
the advantages of the Askey kernel $K_1(x_1,y_1) = \phi_8(|x_1 - y_1|)$ (better stability) 
and the Wendland kernel $K_2(x_2,y_2) = \phi_{1,3}(|x_2 - y_2|)$ (smaller mean square error).

\subsection{Computional efficiency}
We finally discuss the computational efficiency of the proposed tensor methods in Section~\ref{Newton_basis}.
To this end, we consider using a product of two univariate radial characteristic functions to interpolate the bivariate Franke function.
This is done by using four different kernel methods to solve the interpolation problem.

\begin{itemize}
	\item \texttt{standard}: 
	Compute the matrix $A_{K,X}$, then solve the system in (\ref{eq:intpol_system}).
	\item \texttt{kronecker\_prod}: 
	Compute $A_{K,X}$ via the Kronecker product.
	This is done by exploi\-ting the tensor structure as proven by Theorem~\ref{thm:interpoMatrixAsKronecker}.
	Computing the much smaller component matrices $A_{K_i,X^i}$, we get $A_{K,X}$. 
        Then, solve the linear system in (\ref{eq:intpol_system}) as in \texttt{standard}.
	\item \texttt{Newton\_base}: 
	Compute the Newton basis iteratively, update the kernel interpolant at each iteration step.
	\item \texttt{TensorNewton\_Base}: 
	Compute the Newton bases in each component space.
	Exploiting the tensor structure (cf.~Section~\ref{Newton_basis}),
	combine the two bases to obtain the Newton basis of the product kernel. 
	This yields the product kernel interpolant by the solution of the corresponding linear system.
\end{itemize}

We used equidistant $N \times N$ grids in ${\left[ -1, 1 \right] \times \left[ -1, 1 \right]}$, as interpolation points, for varying sizes $N$.
For each of the four above mentioned methods and grid sizes $N$, we recorded the mean computation time (from five runs).
The recorded seconds of CPU times are shown, as functions in $N$, in Figure~\ref{fig:time}.

\begin{figure}[h!]
	\begin{center}
		\fbox{
			\includegraphics[scale=0.3]{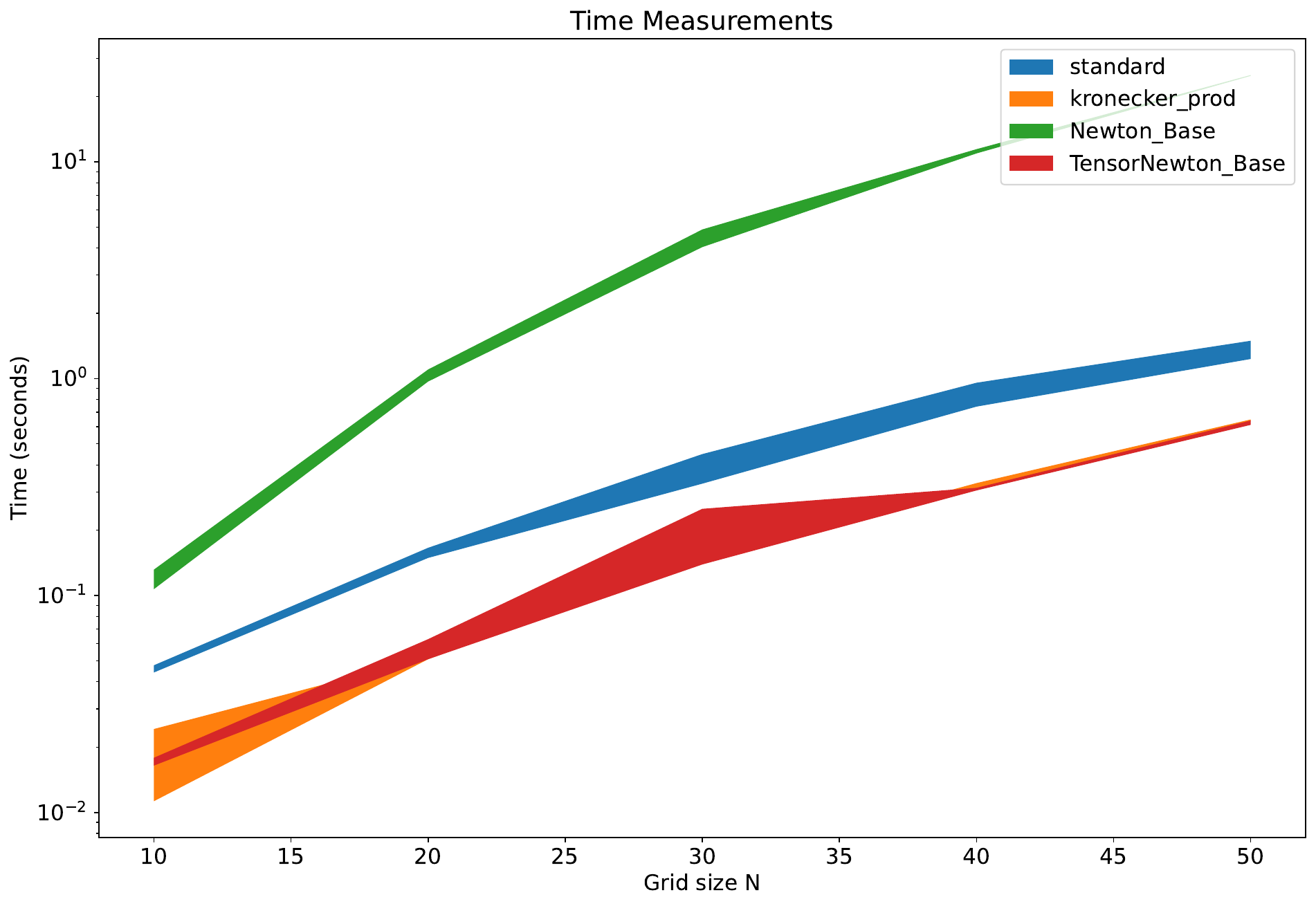}
		}
	\end{center}
	\caption{Computation time (in CPU seconds) as functions of grid size $N$,
	where both the mean values (from 5 test runs) and their standard deviations were recorded,
	as they are displayed by using a standard {\em "fill between plot"}
        (cf.\ {\scriptsize{\tt https://matplotlib.org/stable/api/\_as\_gen/matplotlib.pyplot.fill\_between.html}}).
	}
	\label{fig:time}

\end{figure}

Note from Figure~\ref{fig:time}, that the proposed tensor methods \texttt{kronecker\_prod} and \texttt{TensorNewton\_Base} 
are much more efficient than the corresponding two standard methods \texttt{standard} and \texttt{Newton\_base}.
In fact, the tensor Newton basis method \texttt{TensorNewton\_Base} outperforms the standard Newton basis method \texttt{Newton\_base} quite significantly,
especially for larger data sets. 

In order to further explain on this, we remark that the computational costs for the method \texttt{Newton\_base} 
include the required calculations of the Newton basis for each grid point, just before the interpolant is then being updated. 
This then explains why the method \texttt{Newton\_base} can only be inferior for just {\em one} interpolation update.
In the long run, i.e., for {\em several} interpolation updates, storage of the Newton bases pays off, in which case standard methods would be inferior.

Decomposing the interpolation point set into its univariate components, leads to a significant data reduction, which explains the
superiority of the tensor method with respect to computational complexity.
This gives the proposed tensor product method yet another advantage.

\section{Conclusion and Future Research}
We have proposed product kernels to obtain problem-adapted efficient and flexible tools for high-dimensional scattered data interpolation.
Moreover, we proved by Theorem~\ref{thm:posdef} that tensor products of 
positive definite component kernels 
(that are not necessarily translational-invariant) yield 
positive definite product kernels.
The utility of tensor product kernels is further supported by our numerical results in Section~\ref{numeric}, 
where we have demonstrated not only their efficiency, but also their ability to balance numerical stability against approximation quality.

In view of future research, we wish to propose two possible directions.
Firstly, the generalization of kernel-based approximation methods from the Euclidean space ${\mathbb R}^d$ 
to other metric spaces (e.g.~to Riemannian manifolds), or to even more general topological spaces, has been
studied in various different contexts. 
Machine learning algorithms, for instance, require the construction of approximation schemes for
graphs $\Omega_1 = (V, E)$, with nodes $V$ and edges $E$, and corresponding node features 
in $\Omega_2 \subset {\mathbb R}^d$. In this case, one would like to approximate a target function 
$f : V \times {\mathbb R}^d \longrightarrow {\mathbb R}$ from given values on some of the nodes,
where product kernels are highly relevant. 

Secondly, the applicability of tensor product kernels should be explored for
high dimensions $d = d_1 + \ldots + d_M$, $M \gg 1$, or for infinite dimensional input data, 
where only a few $d_i$'s are contributing to the output. 
To be more concrete, this is relevant in situations, where
(a) target functions are assumed to have a certain low-dimensional structure \cite{rieger2024},
(b) dimensions are down-weighted to zero \cite{fasshauer2012}, or
(c) the input dimensions are selected in a greedy-like fashion \cite{camattari2024}.

\section*{Acknowledgment}
The authors acknowledge the support by the Deutsche Forschungsgemeinschaft (DFG) within the Research Training Group GRK 2583 "Modeling, Simulation and Optimization of Fluid Dynamic Applications". 
We wish to thank Tizian Wenzel for several fruitful discussions. 
His critical remarks and constructive suggestions led to significant improvements over a previous version.
We wish to thank one anonymous referee for many useful suggestions,
especially for those concerning generalizations and the applicability of tensor product kernels. 
The latter inspired us to add Corollary~4.10 and to include a short discussion on future research. 


\end{document}